\documentclass[11pt]{amsart}
\usepackage{amsfonts}
\usepackage{amsmath}
\usepackage{color}
\usepackage[mathscr]{eucal}
\usepackage{amscd, latexsym, graphicx, mathrsfs, enumerate}
\usepackage{amssymb}
\usepackage{bbm}
\usepackage{dsfont}
\usepackage{url}

\makeindex

\newtheorem{thm}{{Theorem}}[section]
\newtheorem{prop}[thm]{{Proposition}}
\newtheorem{lem}[thm]{{Lemma}}
\newtheorem{cor}[thm]{{Corollary}}

\newtheorem{remark}[thm]{Remark}
\numberwithin{equation}{section}
\newtheorem{Def}[thm]{Definition}

\def\N{\mathbb{N}}
\def\Z{\mathbb{Z}}
\def\Q{\mathbb{Q}}
\def\R{\mathbb{R}}
\def\C{\mathbb{C}}
\def\A{\mathbb{A}}

\def\GL{{\mathop{\mathrm{GL}}}}

\def\SL{{\mathop{\mathrm{SL}}}}
\def\O{{\mathop{\mathrm{O}}}}
\def\SO{{\mathop{\mathrm{SO}}}}
\def\U{{\mathop{\mathrm{U}}}}
\def\Sp{{\mathop{\mathrm{Sp}}}}
\def\GSp{{\mathop{\mathrm{GSp}}}}

\def\Sym{{\mathop{\mathrm{Sym}}}}

\def\Re{{\mathop{\mathrm{Re}}}}

\def\Tr{{\mathop{\mathrm{Tr}}}}

\def\diag{{\mathop{\mathrm{diag}}}}

\def\vol{{\mathop{\mathrm{vol}}}}

\def\d{{\mathrm{d}}}
\def\ve{\varepsilon}
\def\tr{{\rm tr}}

\def\bsl{\backslash}
\def\inf{\infty}

\def\min{{\mathrm{min}}}

\def\F{\mathbb{F}}

\def\uk{{\underline{k}}}
\def\ul{{\underline{l}}}

\def\cL{{\mathcal{L}}}

\def\bs{{\backslash}}
\def\ds{\displaystyle}

\def\lra{{\longrightarrow}}
\def\G{{\Gamma}}

\def\trep{{\mathbbm{1}}}

\def\cH{{\mathcal H}}
\def\ur{{\mathrm{ur}}}

\def\tf{{\tilde{f}}}

\numberwithin{equation}{section}

\begin{document}
\title[Siegel cusp forms of general degree]
{Equidistribution theorems for holomorphic Siegel cusp forms of general degree: 
the level aspect}
\author{Henry H. Kim, Satoshi Wakatsuki and Takuya Yamauchi}
\date{\today}

\thanks{The first author is partially supported by NSERC grant \#482564. The second author is partially supported by JSPS Grant-in-Aid for Scientific Research (C) No.20K03565, (B) No.21H00972. The third author is partially supported by 
JSPS KAKENHI Grant Number (B) No.19H01778.}
\address{Henry H. Kim \\
Department of mathematics \\
 University of Toronto \\
Toronto, Ontario M5S 2E4, CANADA \\
and Korea Institute for Advanced Study, Seoul, KOREA}
\email{henrykim@math.toronto.edu}

\address{Satoshi Wakatsuki \\
Faculty of Mathematics and Physics, Institute of Science and Engineering\\
Kanazawa University\\
Kakumamachi, Kanazawa, Ishikawa, 920-1192, JAPAN}
\email{wakatsuk@staff.kanazawa-u.ac.jp}

\address{Takuya Yamauchi \\
Mathematical Inst. Tohoku Univ.\\
 6-3,Aoba, Aramaki, Aoba-Ku, Sendai 980-8578, JAPAN}
\email{takuya.yamauchi.c3@tohoku.ac.jp}

\begin{abstract}
This paper is an extension of \cite{KWY} and we prove equidistribution theorems for families of holomorphic Siegel cusp forms of general degree in the level aspect. Our main contribution is to estimate unipotent contributions for general degree 
in the geometric side of Arthur's invariant trace formula in terms of 
Shintani zeta functions in a uniform way.   
Several applications 
including the vertical Sato-Tate theorem and low-lying zeros for standard $L$-functions of holomorphic Siegel cusp forms 
are discussed. 
We also show that  the ``non-genuine forms" which come from non-trivial 
endoscopic contributions by Langlands functoriality classified by 
Arthur are negligible.
\end{abstract}

\maketitle

\tableofcontents

\section{Introduction}
Let $G$ be a connected reductive group over $\Q$ and $\A$ the ring of adeles of $\Q$. 
An equidistribution theorem for a family of automorphic representations of $G(\A)$ is one of 
recent topics in number theory and automorphic representations. 
After Sauvageot's important results \cite{Sau},
Shin \cite{Shin} proved a so-called limit multiplicity formula which shows that 
the limit of an automorphic counting measure is the Plancherel measure. It implies the 
equidistribution of Hecke eigenvalues or Satake parameters at a fixed prime in a family of cohomological automorphic forms on 
$G(\A)$. 
A quantitative version of Shin's result is given by 
Shin and Templier \cite{ST}. A different approach is discussed in \cite{FLM} for $G=\GL_n$ or $\SL_n$ 
treating more general automorphic forms which are not necessarily cohomological. 
Note that in the works of Shin and Shin-Templier, one needs to consider all cuspidal representations in the $L$-packets. Shin suggested in \cite[the second paragraph in p. 88]{Shin} that one can isolate just holomorphic discrete series at infinity. In \cite{KWY,KWY1}, we carried out his suggestion and established equidistribution theorems for holomorphic Siegel cusp forms of degree 2. We should also mention Dalal's work \cite{Dalal} (cf. Remark \ref{STD}). See also some related works \cite{KL,KSTs}. 

In this paper we generalize several equidistribution theorems to holomorphic Siegel cusp forms of general degree. A main tool is Arthur's invariant trace formula as used in the previous 
work but we need a more careful analysis in the computation of unipotent contributions. 
Let us prepare some notations to explain our results. 

Let $G=\Sp(2n)$ be the symplectic group of rank $n$ defined over $\Q$. 
For an $n$-tuple of integers $\uk=(k_1,\ldots,k_n)$ with $k_1\ge \cdots \ge k_n>n+1$, let 
$D^{{\rm hol}}_{\underline{l}}=\sigma_\uk$ be the holomorphic discrete series representation  of $G(\R)$ with the Harish-Chandra parameter $\underline{l}=(k_1-1,\ldots,k_n-n)$ or 
the Blattner parameter $\uk$. 

Let $\A$ (resp. $\A_f$) be the ring of (resp. finite) adeles of $\Q$, and $\widehat{\Z}$ be the profinite completion of $\Z$.
For $S_1$ a finite set of rational primes, let $S=\{\infty\}\cup S_1$, and $\Q_{S_1}=\prod_{p\in S_1}\Q_p$, and $\A^{S}$ 
be the ring of adeles outside $S$ 
and $\widehat{\Z}^{S}=\prod_{p\not\in S_1}\Z_p$.  
We denote by $\widehat{G(\Q_{S_1})}$ the unitary dual of $G(\Q_{S_1})=\prod_{p\in S_1}G(\Q_p)$ equipped with the Fell topology. 
Fix a Haar measure $\mu^{S}$ on $G(\A^{S})$ so that 
$\mu^{S}(G(\widehat{\Z}^{S}))=1$, and let $U$ be a compact open subgroup of $G(\Bbb A^{S})$. 
Consider the algebraic representation $\xi=\xi_{\underline{k}}$ of the highest weight $\uk$ so that it is isomorphic to the minimal $K_\inf$-type of $D^{\rm hol}_\ul$. 
Let $h_U$ denote the characteristic function of $U$.
Then we define a measure on $\widehat{G(\Q_{S_1})}$ by 
\begin{equation}\label{mu}
\widehat{\mu}_{U,S_1,\xi,D^{\rm hol}_\ul}:
=\frac{1}{{\rm vol}(G(\Q)\bs G(\A))\cdot {\rm dim}\, \xi} \sum_{\pi^0_{S_1}\in \widehat{G(\Q_{S_1})}}
\mu^S(U)^{-1} \, m_{\rm cusp}(\pi^0_{S_1}; U,\xi,D^{\rm hol}_\ul) \, \delta_{\pi^0_{S_1}},
\end{equation} 
where $\delta_{\pi^0_{S_1}}$ is the Dirac delta measure supported at $\pi^0_{S_1}$, a unitary representation of $G(\Q_{S_1})$, and 
\begin{equation}\label{mult}
m_{\rm cusp}(\pi^0_{S_1}; U,\xi,D^{\rm hol}_\ul)=
\sum_{\pi\in \Pi(G(\A))^0\atop \pi_{S_1}\simeq \pi^0_{S_1},\,\pi_\infty\simeq D^{\rm hol}_\ul}m_{\rm cusp}(\pi) \, {\rm tr}(\pi^S(h_U)),
\end{equation} 
where $\Pi(G(\A))^0$ stands for the isomorphism classes of all irreducible unitary cuspidal representations of 
$G(\A)$, and $\pi^S=\otimes'_{p\notin S} \pi_p$.

To state the equidistribution theorem, we need to introduce the Hecke algebra $C^\infty_c(G(\Q_{S_1}))$ which is dense 
under the map $h\mapsto \widehat{h}$, where $\widehat{h}(\pi_{S_1})=
{\rm tr}(\pi_{S_1}(h))$ is in 
 $\mathcal{F}(\widehat{G(\Q_{S_1})})$ consisting of suitable $\widehat{\mu}^{{\rm pl}}_{S_1}$-measurable functions 
on $\widehat{G(\Q_{S_1})}$. (See \cite[Section 2.3]{Shin} for that space.)

Let $N$ be a positive integer. Put $S_N=\{p\ {\rm prime} :\ p|N\}$. We assume that $S_1\cap S_N=\emptyset$.  
We denote by $K_p(N)$ the the principal congruence subgroup of level $N$ for $G(\Z_p)$ (see \eqref{eq:prin} for the definition), and set $K^S(N)=\prod_{p\notin S}K_p(N)$. 
For each rational prime $p$, let us consider 
the unramified Hecke algebra $\mathcal{H}^{{\rm ur}}(G(\Q_p))\subset C^\infty_c(\Q_p)$ and 
for each $\kappa >0$, $\mathcal{H}^{{\rm ur}}(G(\Q_p))^\kappa$, the linear subspace of 
$\mathcal{H}^{{\rm ur}}(G(\Q_p))$ consisting of all Hecke elements whose heights are less than $\kappa$.  
(See (\ref{height}).) Let $\mathcal{H}^{{\rm ur}}(G(\Q_p))^\kappa_{\le 1}$ be the linear subspace of 
$\mathcal{H}^{{\rm ur}}(G(\Q_p))^\kappa$ consisting of all Hecke elements whose complex values have absolute values less than 1.
Our first main result is 

\begin{thm}\label{main1} Fix $\uk=(k_1,\ldots,k_n)$ satisfying $k_1\ge \cdots \ge k_n>n+1$. Fix a positive integer $\kappa$. 
Then there exist constants $a,b$, and $c_0\neq 0$  depending only on $G$ such that 
for each $h_1\in \otimes_{p\in S_1} \mathcal{H}^{{\rm ur}}(G(\Q_p))^\kappa_{\le 1}$, 

$$\widehat{\mu}_{K^S(N),S_1,\xi,D^{\rm hol}_{\ul}}(\widehat{h_1})=\widehat{\mu}^{{\rm pl}}_{S_1}\left(\widehat{h_1}\right)+
O\left(\left(\prod_{p\in S_1}p\right)^{a\kappa+b} N^{-n}\right),
$$
if $N\ge c_0 \prod_{p\in S_1}p^{2n\kappa}$. Note that the implicit constant of the Landau $O$-notation is independent of $S_1$, $N$ and $h_1$.  
\end{thm}

Let us apply this theorem to the vertical Sato-Tate theorem   
and higher level density theorem for standard $L$-functions of holomorphic Siegel cusp forms.

The principal congruence subgroup $\Gamma(N)$ of level $N$ for $G(\Z)$ is obtained by $\Gamma(N)=G(\Bbb Q)\cap G(\Bbb R)K(N)$ where $K(N)=\prod_{p<\inf} K_p(N)$. 
Let $S_\uk(\G(N))$ be the space of holomorphic Siegel 
cusp forms of weight $\uk$ with respect to $\Gamma(N)$ (see the next section for a precise definition), and let 
$HE_\uk(N)$ be a basis consisting of all Hecke eigenforms outside $N$. We can identify $HE_\uk(N)$ with a basis of $K(N)$-fixed vectors in the set of cuspidal representations of $G(\Bbb A)$ whose infinity component is (isomorphic to) $D^{\rm hol}_{\ul}$. (See the next section for the detail.)
Put $d_\uk(N)=|HE_\uk(N)|$. Then we have \cite{Wakatsuki}, for some constant $C_\uk>0$, 
\begin{equation}\label{dimension}
d_{\uk}(N)=C_{\uk} C_N N^{2n^2+n}+O_{\uk}(N^{2n^2}),
\end{equation}
where $\ds C_N=\prod_{p|N} \prod_{i=1}^n (1-p^{-2i})$. Note that $\ds\prod_{i=1}^n \zeta(2i)^{-1}< C_N< 1$.

For each $F\in HE_\uk(N)$, we denote by 
$\pi_F=\otimes'_p\pi_{F,p}$ the corresponding automorphic cuspidal representation of $G(\A)$. 
Henceforth we assume that 
\begin{equation}\label{suff-reg}
k_1>\cdots>k_n>n+1.
\end{equation} 
Then the Ramanujan conjecture is true, namely, $\pi_{F,p}$ is tempered for any $p$ 
(see Theorem \ref{Ram}). 
Unfortunately, this assumption forces us to exclude the scalar-valued Siegel cusp forms.

Let $\widehat{G(\Q_p)}^{{\rm ur,temp}}$ be the subspace of $\widehat{G(\Q_p)}$ consisting of 
all unramified tempered classes. We denote by $(\theta_1(\pi_{F,p}),\ldots,\theta_n(\pi_{F,p}))$ the element of $\Omega$ corresponding to $\pi_{F,p}$ 
under the isomorphism
$\widehat{G(\Q_p)}^{{\rm ur,temp}}\simeq [0,\pi]^n/\mathfrak S_n=:\Omega$. Let $\mu_p$ be the measure on $\Omega$ defined in Section \ref{vertical}.

\begin{thm}\label{Sato-Tate-tm} Assume (\ref{suff-reg}). Fix a prime $p$.
Then the set  $\left\{(\theta_1(\pi_{F,p}),\ldots,\theta_n(\pi_{F,p}))\in \Omega\ |\ F\in HE_\uk(N) \right\}$ is 
$\mu_p$-equidistributed in $\Omega$, namely, for each continuous function $f$ on $\Omega$, 
$$\lim_{N\to \infty \atop (p,N)=1}\frac{1}{d_\uk(N)}\sum_{F\in HE_\uk(N)}f(\theta_1(\pi_{F,p}),\ldots,\theta_n(\pi_{F,p}))=
\int_{\Omega}f(\theta_1,\ldots,\theta_n)\mu_p.
$$
\end{thm}

By using Arthur's endoscopic classification, we have a finer version of the above theorem.
Under the assumption (\ref{suff-reg}), the global $A$-parameter describing $\pi_F,\ F\in HE_\uk(N)$, 
is always semi-simple. (See Definition \ref{cap-endo}.) Let $HE_\uk(N)^{{g}}$ be the subset of $HE_\uk(N)$ consisting of $F$ such that the global $A$-packet 
containing $\pi_F$ is associated to a simple global $A$-parameter. They are Siegel cusp forms which do not come from smaller 
groups by Langlands functoriality in Arthur's classification. In this paper, we call them genuine forms. Let $HE_\uk(N)^{{ng}}$ be the subset of $HE_\uk(N)$ consisting of $F$ such that the global $A$-packet 
containing $\pi_F$ is associated to a non-simple global $A$-parameter, i.e., they are Siegel cusp forms which come from smaller 
groups by Langlands functoriality in Arthur's classification. We call them non-genuine forms. We show that non-genuine forms are negligible. For this, we need some further assumptions on the level $N$. 

\begin{thm}\label{finer-ver} Assume (\ref{suff-reg}). We also assume:
\begin{enumerate} 
\item $N$ is an odd prime or 
\item $N$ is odd and all prime divisors $p_1,\ldots,p_r\ (r\ge 2)$ of $N$ are 
congruent to 1 modulo 4 such that 
$\Big(\ds\frac{p_i}{p_j}\Big)=1$ for $i\ne j$,
where $\Big(\ds\frac{\ast}{\ast}\Big)$ denotes the Legendre symbol. 
\end{enumerate}
Then 
\begin{enumerate}
\item $\left|HE_\uk(N)^{{g}}\right|=C_{\uk}C_N N^{2n^2+n}+O_{n,\uk,\epsilon}\left(N^{2n^2+n-1+\epsilon}\right)$ for any $\epsilon>0$; 
\item $\left|HE_\uk(N)^{{ng}}\right|=O_{n,\uk,\epsilon}\left(N^{2n^2+n-1+\epsilon}\right)$ for any $\epsilon>0$;
\item for a fixed prime $p$, the set $\left\{(\theta_1(\pi_{F,p}),\ldots,\theta_n(\pi_{F,p}))\in \Omega\ |\ F\in HE_\uk(N)^{{g}} \right\}$ is 
$\mu_p$-equidistributed in $\Omega$. 
\end{enumerate}
\end{thm}  

The above assumptions on the level $N$ are necessary in order to estimate non-genuine 
forms related to non-split but quasi split orthogonal groups in the Arthur's classification  by using the transfer theorems for some Hecke elements in
 the quadratic base change in the ramified case \cite{Yamauchi}. (See Proposition \ref{fixedV} for the detail.)

Next, we discuss $\ell$-level density ($\ell$ positive integer) for standard L-functions in the level aspect. 
Let us denote by $ \Pi({\rm GL}_{n}(\A))^0$ the set of all isomorphism classes of 
irreducible unitary cuspidal representations of ${\rm GL}_{n}(\A)$. 
Keep the assumption on $\uk$ as in (\ref{suff-reg}) and the above assumption on the level $N$. Then $F$ can be described by a global $A$-parameter $\boxplus_{i=1}^r\pi_i$ 
with $\pi_i\in \Pi({\rm GL}_{m_i}(\A))^0$ and $\ds\sum_{i=1}^rm_i=2n+1$. 
Then we may define the standard $L$-function of $F\in HE_\uk(N)$ by 
$$L(s,\pi_F,{\rm St}):=\prod_{i=1}^r L(s,\pi_i)$$
which coincides with the classical definition in terms of Satake parameters of $F$ outside $N$. 
Then we show unconditionally that the $\ell$-level density of the standard $L$-functions of the family 
$HE_\uk(N)$ has the symmetry type $Sp$ in the level aspect. (See Section \ref{r-level} for the precise statement. Shin and Templier \cite{ST} showed it under several hypotheses for a family which includes non-holomorphic forms.)
Here in order to obtain lower bounds for conductors, it is necessary to introduce a concept of newforms. This may be of independent interest. 
Since any local newform theory for $\Sp(2n)$ is unavailable except for $n=1,2$, we define the old space
$S^{{\rm old}}_\uk(\Gamma(N))$ to be the intersection of 
$S_\uk(\Gamma(N))$ with the
smallest $G(\A_f )$-invariant space of functions on $G(\Q)\bs G(\A)$ containing 
$S_\uk(\Gamma(M))$
for all proper divisors $M$ of $N$. The new space $S^{{\rm new}}_\uk(\Gamma(N))$ is the 
 orthogonal complement of $S^{{\rm old}}_\uk(\Gamma(N))$ in  
$S_\uk(\Gamma(N))$ with respect to Petersson inner product. 
Then if $F\in S^{{\rm new}}_\uk(\Gamma(N))$, $q(F)\geq N^\frac 12$ (Theorem \ref{stan-conductor}), and if $N$ is square free, we can show that $\text{dim}\, S^{{\rm new}}_\uk(\Gamma(N))\geq \zeta(n^2)^{-1} d_\uk(N)$ if $n\geq 2$ (Theorem \ref{newforms}).

As a corollary, we obtain a result on the order of vanishing of $L(s,\pi_F,{\rm St})$ at $s=\frac 12$, the center of symmetry of the $L$-function by using the method of 
\cite{ILS} for holomorphic cusp forms on $\GL_2(\A)$ (see also \cite{Brumer} for another formulation related to Birch-Swinnerton-Dyer conjecture): Let $r_F$ be the order of vanishing of $L(s,\pi_F,{\rm St})$ at $s=\frac 12$.
Then we show that under GRH (generalized Riemann hypothesis), $\displaystyle\sum_{F\in HE_\uk(N)} r_F\leq C d_\uk(N)$ for some constant $C>0$. This would be the first result of this kind in Siegel modular forms.
We can also show a similar result for the degree 4 spinor $L$-functions of $\GSp(4)$.

Let us explain our strategy in comparison with the previous works. 
We choose a test function
\[
f=\mu^S(K(N))^{-1}f_\xi h_1 h_{K^S(N)}\in C^\infty_c(G(\R))\otimes \left( \otimes_{p\in S_1} \cH^\ur (G(\Q_p))^\kappa_{\le 1} \right) \otimes  C^\infty_c(G(\A^S))
\]
such that $f_\xi$ is a pseudo-coefficient of $D^{\rm hol}_\ul$ normalized as $\tr(\pi_\infty(f_\xi))=1$.
A starting main equality is 
$$
I_{{\rm spec}}(f)=I(f)=I_{{\rm geom}}(f),
$$ 
where $I_{{\rm spec}}(f)$ (resp. $I_{{\rm geom}}(f)$) is the spectral (resp. the geometric) side of the 
Arthur's invariant trace $I(f)$. 
Under the assumption $k_n>n+1$, the spectral side becomes simple by 
the results of \cite{Arthur} and \cite{Hiraga}, and it is directly 
related to $S_\uk(\Gamma(N))$ because of the choice of a pseudo-coefficient of $D^{{\rm hol}}_\ul$.
Now the geometric side is given by:
\begin{equation}\label{geom}
I_{{\rm geom}}(f)=\sum_{M\in\cL}(-1)^{\dim(A_M/A_G)}\frac{|W_0^M|}{|W_0^G|}
\sum_{\gamma\in (M(\Q))_{M,\tilde{S}}}a^M(\tilde{S},\gamma)\, I_M^G(\gamma,f_\xi)\, J_M^M(\gamma,h_P),
\end{equation}
($\tilde{S}=\{\inf\}\sqcup S_N \sqcup S_1$), where 
$(M(\Q))_{M,\tilde{S}}$ denotes the set of $(M,\tilde{S})$-equivalence classes in $M(\Q)$ (cf. \cite[p.113]{Arthur2}); 
for each $M$ in a finite set $\cL$ we choose a parabolic subgroup $P$ such that $M$ is a Levi subgroup of $P$. (See loc.cit. for details.) 
Roughly speaking,
\begin{itemize}

\item If the test function $f$ is fixed, the terms on \eqref{geom} vanish except for a finite number of $(M,\tilde{S})$-equivalence classes.

\item the factor $a^M(\tilde{S},\gamma)$ is called a global coefficient and it is almost the volume of the centralizer of $\gamma$ in $M$ if $\gamma$ is semisimple. The general properties are unknown;

\item the factor $I_M^G(\gamma,f_\xi)$ is called an invariant weighted orbital integral, and as the notation shows, it strongly depends on the weight $\uk$ of $\xi=\xi_\uk$. Therefore, it is negligible when we consider the level aspect;
   
\item the factor $J_M^M(\gamma,h_P)$ is an orbital integral of $\gamma$ for $h=\mu^S(K(N))^{-1}h_1h_{K^S(N)}$.

\end{itemize} 
According to the types of conjugacy classes and $M$, the geometric side is divided into the following terms:
$$I_{{\rm geom}}(f)=I_1(f)+I_2(f)+I_3(f)+I_4(f)
$$
where 
\begin{itemize}
\item $I_1(f)$: $M=G$ and $\gamma=1$;
\item $I_2(f)$: $M\neq G$ and $\gamma=1$;
\item $I_3(f)$: $\gamma$ is unipotent but $\gamma\neq 1$;
\item $I_4(f)$: the other contributions. 
\end{itemize}

The first term $I_1(f)$ is $f(1)$ up to constant factors and the Plancherel formula 
$\widehat{\mu}^{{\rm pl}}_{S_1}(\widehat{f})=f(1)$ yields the first term of the equality in Theorem \ref{main1}. 
The condition $N\ge c_0 \prod_{p\in S_1}p^{2n\kappa}$ in Theorem \ref{main1} implies that the non-unipotent contribution $I_4(f)$ vanishes by \cite[Lemma 8.4]{ST}.
Therefore, everything is reduced to studying the unipotent contributions $I_2(f)$ and $I_3(f)$. 
An explicit bound for $I_2(f)$ was given by \cite[Proof of Theorem 9.16]{ST}.
However, as for $I_3(f)$, since the number of $(M,\tilde{S})$-equivalence classes in the geometric unipotent conjugacy class of each $\gamma$ is increasing when $N$ goes to infinity, it is difficult to estimate $I_3(f)$ directly.
In the case of $\GSp(4)$, we computed unipotent contributions by using case by case analysis in \cite{KWY}.
Here we give a new uniform way to estimate all the unipotent contributions. 
It is given by a sum of special values of zeta integrals with real characters for spaces of symmetric matrices (see Lemma \ref{lem:210513l2} and Theorem \ref{thm:tr}).
This formula is a generalization of the dimension formula (cf. \cite{Shintani, Wakatsuki}) to the trace formula of Hecke operators. 
By using their explicit formulas \cite{Saito2} and analyzing Shintani double zeta functions \cite{KTW}, we express the geometric side as a finite sum of products of local integrals and special values of the Hecke $L$ functions with real characters, and then obtain the estimates of the geometric side (see Theorem \ref{thm:asy}). 

This paper is organized as follows. In Section 2, we set up some notations. 
In Section 3, we give key results (Theorem \ref{thm:tr}) and Theorem \ref{thm:asy}) 
in estimating trace formulas of Hecke elements. In Section 4, we study Siegel modular forms in terms of Arthur's classification and show that non-genuine forms are negligible. In Section 5, we give a notion of newforms which is necessary to estimate conductors.
Section 6 through 10 are devoted to proving the main theorems. Finally, in the Appendix, we give 
an explicit computation of the convolution product of some Hecke elements which is needed in 
the computation of $\ell$-level density of standard $L$-functions. 

\medskip

\textbf{Acknowledgments.} We would like to thank M. Miyauchi, M. Oi, S. Sugiyama, and M. Tsuzuki for helpful discussions. We thank KIAS in Seoul and RIMS in Kyoto for their incredible hospitality during this research. We thank the referees for their helpful remarks and corrections. 

\medskip

\section{Preliminaries}\label{sec:pre}

A split symplectic group $G=\Sp(2n)$ over the rational number field $\Q$ is defined by
\[
G=\Sp(2n)=\left\{g\in\GL_{2n} \Bigg|\ g\begin{pmatrix} O_n&I_n \\ -I_n & O_n \end{pmatrix} {}^t\!g=\begin{pmatrix} O_n&I_n \\ -I_n & O_n \end{pmatrix}  \right\}.
\]
The compact subgroup
\[
K_\inf=\left\{ \begin{pmatrix}A&-B \\ B & A \end{pmatrix}\in G(\R) \right\}
\]
of $G(\Bbb R)$ is isomorphic to the unitary group $\U(n)$ via the mapping  $ \begin{pmatrix}A&-B \\ B & A \end{pmatrix} \mapsto 
A + i B$, where $i=\sqrt{-1}$. For each rational prime $p$, 
we also set $K_p=G(\Z_p)$ and put $K=\prod_{p\le \infty} K_p$. The compact groups $K_v$ and $K$ are maximal in $G(\Q_v)$ and $G(\A)$, resp.

Holomorphic discrete series of $G(\R)$ are parameterized by $n$-tuples $\underline{k}=(k_1,\dots,k_n)\in \Z^n$ such that $k_1\geq\cdots\geq k_n>n$, which is called the Blattner parameter.
We write $\sigma_\uk$ for the holomorphic discrete series corresponding to the Blattner parameter $\uk=(k_1,\ldots,k_n)$. We also write $D^{{\rm hol}}_\ul$ for one corresponding to the Harish-Chandra parameter $\ul=(k_1-1,k_2-2,\ldots,k_n-n)$ so that $D^{{\rm hol}}_\ul=\sigma_{\uk}$. 

Let $\cH^\ur(G(\Q_p))$ denote the unramified Hecke algebra over $G(\Q_p)$, that is,
\[
\cH^\ur(G(\Q_p))=\{\varphi\in C_c^\inf(G(\Q_p)) \mid \varphi(k_1xk_2)=\varphi(x) \quad (\forall k_1,k_2\in K_p , \;\; \forall x\in G(\Q_p)) \}.
\]
Let $T$ denote the maximal split $\Q$-torus of $G$ consisting of diagonal matrices.
We denote by $X_*(T)$ the group of cocharacters on $T$ over $\Q$.
An element $e_j$ in $X_*(T)$ is defined by
\begin{equation}\label{ejx}
e_j(x)=\diag(\overbrace{1,\dots,1}^{j-1},x,\overbrace{1\dots,1}^{n-j+1},\overbrace{1,\dots,1}^{j-1},x^{-1},\overbrace{1,\dots,1}^{n-j+1})\in T \quad (x\in \mathbb{G}_m).
\end{equation}
Then, one has $X_*(T)=\langle e_1,\dots,e_n \rangle$.
By the Cartan decomposition, any function in $\cH^\ur(G(\Q_p))$ is expressed by a linear combination of characteristic functions of double cosets $K_p\lambda(p)K_p$ $(\lambda\in X_*(T))$.
A height function $\|\; \|$ on $X_*(T)$ is defined by
\[
\left\|\prod_{j=1}^n e_j^{m_j}\right\|=\max\{ |m_j| \mid 1\leq j\leq n\} , \qquad (m_j\in\Z).
\]
For each $\kappa\in\N$, we set
\begin{equation}\label{height}
\cH^\ur(G(\Q_p))^\kappa =\Big\{ \varphi\in \cH^\ur(G(\Q_p)) \mid \mathrm{Supp}(\varphi) \subset \bigcup_{\mu\in X_*(T), \; \|\mu\|\leq \kappa } K_p\mu(p)K_p \Big\}.
\end{equation}

Choose a natural number $N$.
We set
\begin{equation}\label{eq:prin}
K_p(N)=\{x\in K_p \mid x\equiv I_{2n} \mod N\}, \quad  K(N)= \prod_{p<\inf} K_p(N) .
\end{equation}
One gets a congruence subgroup $\Gamma(N)=G(\Q)\cap G(\R)K(N)$.

Let $\mathfrak{H}_n:=\{Z\in M_n(\C) \mid Z={}^t\!Z, \; \mathrm{Im}(Z)>0  \}$.
We write $S_\uk(\Gamma(N))$ for the space of Siegel cusp forms of weight $\uk$ for $\Gamma(N)$, i.e., $S_\uk(\Gamma(N))$ consists of $V_\uk$-valued smooth functions $F$ on $G(\A)$ satisfying the following conditions:
\[
\begin{array}{ll}
\text{(i)}& F( \gamma gk_\inf k_f)=\rho_\uk(k_\inf)^{-1}F(g), \; g\in G(\A) , \;  \gamma\in G(\Q), \; k_\inf \in K_\inf, \; k_f\in K(N)    \\
\text{(ii)}& \text{$ \rho_\uk(g_\inf,iI_n) F|_{G(\R)}(g_\inf)$ is holomorphic for $g_\inf\cdot iI_n \in \mathfrak{H}_n$,} \\
\text{(iii)}& \max_{g\in G(\A)}\left|F(g)\right|\ll 1,
\end{array} 
\]
where $\rho_\uk$ denotes the finite dimensional irreducible polynomial representation of $\U(n)$ corresponding to $\uk$ together with the representation space $V_\uk$, and we set $\rho_\uk(g,iI_n)=\rho_\uk(iC+D)$ for 
$g=\left(\begin{smallmatrix}A&B \\ C&D\end{smallmatrix}\right)\in G(\R)$.

Let $\underline{m}=(m_1,...,m_n)$, $m_1| m_2|\cdots |m_n$, and $D_{\underline{m}}=\diag(m_1,...,m_n)$. 
Let $T(D_{\underline{m}})$ be the Hecke operator defined by the double coset $\Gamma(N)\begin{pmatrix} D_{\underline{m}}&0\\0&D_{\underline{m}}^{-1}\end{pmatrix}\Gamma(N)$. 
In particular, for each prime $p$, let
$D_{p,\underline{a}}=\diag(p^{a_1},...,p^{a_n})$, with $\underline{a}=(a_1,...,a_n), 0\leq a_1\leq\cdots\leq a_n$.

Let $F$ be a Hecke eigenform in $S_{\uk}(\Gamma(N))$ with respect to the Hecke operator $T(D_{p,\underline{a}})$ 
for all $p\nmid N$. (See \cite[Section 2.2]{KWY} for Hecke eigenforms  in the case of 
$n=2$. One can generalize the contents there to $n\geq 3$.)
Then $F$ gives rise to an adelic automorphic form $\phi_F$ on $\Sp(2n,\Bbb Q)\backslash\Sp(2n,\Bbb A)$, and $\phi_F$ gives rise to a cuspidal representation $\pi_F$ which is a direct sum $\pi_F=\pi_1\oplus\cdots\oplus \pi_r$, where $\pi_i$'s are irreducible cuspidal representations of $\Sp(2n)$. Since $F$ is an eigenform, $\pi_i$'s are all near-equivalent to each other.
Since we do not have the strong multiplicity one theorem for $\Sp(2n)$, we cannot conclude that $\pi_F$ is irreducible. However, the strong multiplicity one theorem for $\GL_n$ implies that there exists a global $A$-parameter $\psi\in\Psi(G)$ such that $\pi_i\in\Pi_\psi$ for all $i$ \cite[p. 3088]{Schm}. (See Section \ref{Arthur} for the definition of the global $A$-packet.)

On the other hand, given a cuspidal representation $\pi$ of $\Sp(2n)$ with a $K(N)$-fixed vector and whose infinity component is holomorphic discrete series of lowest weight $\uk$, there exists a holomorphic Siegel cusp form $F$ of weight $\uk$ with respect to $\Gamma(N)$ such that $\pi_F=\pi$. (See \cite[p. 2409]{Schm1} for $n=2$. One can generalize the contents there to $n\geq 3$.)

So we define $HE_\uk(N)$ to be a basis of $K(N)$-fixed vectors in the set of cuspidal representations of $\Sp(2n,\Bbb A)$ whose infinity component is holomorphic discrete series of lowest weight $\uk$, and identify it with a basis consisting of all 
Hecke eigenforms outside $N$. In particular, each $F\in HE_\uk(N)$ gives rise to an irreducible cuspidal representation $\pi_F$ of $\Sp(2n)$.
Let $\mathcal F_\uk(N)$ be the set of all isomorphism classes of cuspidal representations of $\Sp(2n)$ such that $\pi^{K(N)}\ne 0$ and $\pi_\infty\simeq \sigma_\uk$.
Consider the map $\Lambda: HE_\uk(N)\longrightarrow \mathcal F_\uk(N)$, given by $F\longmapsto \pi_F$. It is clearly surjective.   
For each $\pi=\pi_\infty\otimes \otimes_p' \pi_p\in \mathcal F_\uk(N)$, set $\pi_f=\otimes_p' \pi_p$. Then we get $|\Lambda^{-1}(\pi)|=\dim \pi_f^{K(N)}$, where $\pi_f^{K(N)}=\{ \phi\in \pi_f \mid \pi_f(k)\phi=\phi \,$ for all $k\in K(N)\}$. 

\section{Asymptotics of Hecke eigenvalues}

For each function $h\in C_c^\inf(K(N)\bsl G(\A_f)/K(N))$, an adelic Hecke operator $T_h$ on $S_\uk(\Gamma(N))$ is defined by
\[
(T_h F)(g)=\int_{G(\A_f)} F(g x) h(x) \, \d x, \qquad F\in S_\uk(\Gamma(N)).
\]
See \cite[p.15--16]{KWY} for the relationship between the classical Hecke operators and adelic Hecke operators for $n=2$. One can generalize the contents there to $n\geq 3$ easily. 
Let $f_\uk$ denote a pseudo-coefficient of $\sigma_\uk$ with ${\rm tr} \sigma_\uk(f_\uk)=1$ (cf. \cite{CD}). 
\begin{lem}\label{lem:20210602sp}
Suppose $k_n>n+1$ and $h\in C_c^\inf(K(N)\bsl G(\A_f)/K(N))$.
The spectral side $I_\mathrm{spec}(f_\uk h)$ of the invariant trace formula is given by 
\[
I_\mathrm{spec}(f_\uk h)=\sum_{\pi=\sigma_\uk\otimes\pi_f,\; \mathrm{auto. \, rep. \, of} \, G(\A)}m_\pi\,  \Tr(\pi_f(h))= \Tr\left( T_h|_{S_\uk(\Gamma(N))}\right),
\]
where $m_\pi$ means the multiplicity of $\pi$ in the discrete spectrum of $L^2(G(\Q)\bsl G(\A))$. 
\end{lem}
\begin{proof}
The second equality follows from \cite{Wallach}. 
One can prove the first equality by using the arguments in \cite{Arthur} and the main result in \cite{Hiraga}.
\end{proof}

We choose two natural numbers $N_1$ and $N$, which are mutually coprime. 
Suppose that $N_1$ is square free. 
Set $S_1=\{  p\,:\,  p|N_1\}$.
We write $h_N$ for the characteristic function of $\prod_{p\notin S_1\sqcup\{\inf\}} K_p(N)$.
For each automorphic representation $\pi=\pi_\infty\otimes\otimes_p' \pi_p$, we set $\pi_{S_1}=\otimes_{p\in S_1} \pi_p$.
\begin{lem}\label{lem:210513l1}
Take a test function $h$ on $G(\A_f)$ as
\begin{equation}\label{eq:testft}
h=\vol(K(N))^{-1}\times  h_1 \otimes h_N, \quad h_1\in \otimes_{p\in S_1}\cH^\ur(G(\Q_p)).
\end{equation}
Then
\[
I_\mathrm{spec}(f_\uk h)=\sum_{\pi=\sigma_\uk\otimes\pi_f,\; \mathrm{auto. \, rep. \, of} \, G(\A)}m_\pi\,  \dim\pi_f^{K(N)} \, \Tr(\pi_{S_1}(h_1))= \Tr\left( T_h|_{S_\uk(\Gamma(N))}\right).
\]
\end{lem}
\begin{proof}
This lemma immediately follows from Lemma \ref{lem:20210602sp}.
\end{proof}

Let $V_r$ denote the vector space of symmetric matrices of degree $r$, and define a rational representation $\rho$ of the group $\GL_1\times \GL_r$ on $V_r$ by $x\cdot \rho(a,m)=a{}^t\!mxm$ $(x\in V_r$, $(a,m)\in \GL_1\times \GL_r)$.
The kernel of $\rho$ is given by $\mathrm{Ker}\rho=\{  (a^{-2} , aI_r) \mid a \in \GL_1 \}$, and we set $H_r=\mathrm{Ker}\rho\bsl (\GL_1 \times \GL_r)$.
Then, the pair $(H_r,V_r)$ is a prehomogeneous vector space over $\Q$.
For $1\leq r\leq n$ and $f\in C_c^\inf(G(\A))$ $\left(\text{resp.}\, f\in C_c^\inf(G(\A_f))\right)$, we define a function $\Phi_{f,r}\in C_c^\inf(V_r(\A))$ $\left(\text{resp.}\,  \Phi_{f,r}\in C_c^\inf(V_r(\A_f))\right)$ as
\[
\Phi_{f,r}(x)=\int_K f(k^{-1}\begin{pmatrix} I_n & * \\ O_n & I_n \end{pmatrix}k) \, \d k  \quad \left(\text{resp.} \;\; \int_{K_f} \right),\quad *=\begin{pmatrix} x&0 \\ 0 &0 \end{pmatrix}\in V_n.
\]

Let $\tf_\uk$ denote the spherical trace function of $\sigma_\uk$ with respect to $\rho_\uk$ on $G(\R)$ (cf. \cite[\S 5.3]{Wakatsuki}). 
Notice that $\tf_\uk$ is a matrix coefficient of $\sigma_\uk$, and so it is not compactly supported.
Take a test function $h\in C_c^\inf(G(\A_f))$ and set $\tf=\tf_\uk h$.
Let $\chi$ be a real character on $\R_{>0}\Q^\times\bsl \A^\times$.
Define a zeta integral $Z_r(\Phi_{\tf,r},s,\chi)$ by
\[
Z_r(\Phi_{\tf,r},s,\chi)=\int_{H_r(\Q)\bsl H_r(\A)}|a^r\det(m)^2|^{s} \chi(a) \, \sum_{x\in V_r^0(\Q)} \Phi_{\tf,r}(x\cdot g) \, \d g \quad (g=\rho \,(a,m)),
\]
where $V_r^0=\{x\in V_r\mid \det(x)\neq 0\}$ and $\d g$ is a Haar measure on $H_r(\A)$.
The zeta integral $Z_r(\Phi_{\tf,r},s,\chi)$ is absolutely convergent for the range 
\begin{equation}\label{eq:range0605}
k_n>2n, \quad   \mathrm{Re}(s)>\frac{r-1}{2}, \quad \begin{cases} \Re(s)<\frac{k_n}{2} & \text{if $r=2$},  \\ \mathrm{Re}(s)<k_n-\frac{r-1}{2} & \text{otherwise}, \end{cases}
\end{equation}
see \cite[Proposition 5.15]{Wakatsuki}, and $Z(\Phi_{\tf,r},s,\chi)$ is meromorphically continued to the whole $s$-plane (cf. \cite{Shintani,Wakatsuki,Yukie}).  
The following lemma associates $Z(\Phi_{\tf,r},s,\chi)$ with the unipotent contribution $I_\mathrm{unip}(f)=I_1(f)+I_2(f)+I_3(f)$ of the invariant trace formula.
\begin{lem}\label{lem:210513l2}
Let $S_0$ be a finite set of finite places of $\Q$.
Take a test function $h_{S_0}\in C_c^\inf (G(\Q_{S_0}))$, and let $h^{S_0}$ denote the characteristic function of $\prod_{p\notin S_0\sqcup\{\inf\}}K_p$. 
Define a test function $\tf$ as $\tf=\tf_\uk h_{S_0}h^{S_0}$.
If $k_n$ is sufficiently large $(k_n \gg 2n)$, then we have
\[
I_\mathrm{unip}(f_\uk h_{S_0}h^{S_0})=\vol_G\, h_{S_0}(1)\, d_\uk  + \frac{1}{2}\sum_{r=1}^n \sum_{\chi\in \mathscr{X}(S_0)} Z_r(\Phi_{\tf,r},n-\tfrac{r-1}{2},\chi),
\]
where $\vol_G=\vol(G(\Q)\bsl G(\A))$, $d_\uk$ denotes the formal degree of $\sigma_\uk$, and $\mathscr{X}(S_0)$ denotes the set consisting of real characters $\chi=\otimes_v \chi_v$ on $\R_{>0}\Q^\times\bsl \A^\times$ such that $\chi_v$ is unramified for any $v\notin S_0\sqcup\{\inf\}$.
\end{lem}
\begin{remark}
Note that the point $s=n-\frac{r-1}{2}$ $(1\leq r\leq n)$ is contained in the range \eqref{eq:range0605}, and we have $Z_r(\Phi_{\tf,r},s,\chi)\equiv 0$ for any real character $\chi\notin \mathscr{X}(S_0)$.
\end{remark}
\begin{proof}
To study $I_\mathrm{unip}(f_\uk h_{S_0}h^{S_0})$, we need an additional zeta integral $\tilde{Z}_r(\Phi_{\tf,r},s)$ defined by
\[
\tilde{Z}_r(\Phi_{\tf,r},s)=\int_{\GL_r(\Q)\bsl \GL_r(\A)}|\det(m)|^{2s} \sum_{x\in V_r^0(\Q)} \Phi_{\tf,r}( {}^t\!m x m) \, \d m.
\]
The zeta integral $\tilde{Z}_r(\Phi_{\tf,r},s)$ is absolutely convergent for the range \eqref{eq:range0605}, and $\tilde{Z}(\Phi_{\tf,r},s)$ is meromorphically continued to the whole $s$-plane, see \cite{Shintani,Wakatsuki,Yukie}.
Applying \cite[Proposition 3.8, Proposition 3.11, Lemmas 5.10 and 5.16]{Wakatsuki} to $I_\mathrm{unip}(f)$, we obtain
\begin{equation}\label{eq:1030}
I_\mathrm{unip}(f_\uk h_{S_0}h^{S_0})=\vol_G\, h_{S_0}(1)\, d_\uk  + \sum_{r=1}^n  \tilde{Z}_r(\Phi_{\tf,r},n-\tfrac{r-1}{2})
\end{equation}
for sufficiently large $k_n\gg 2n$.
Notice that $f_\uk$ is changed to $\tf_\uk$ in the right-hand side of \eqref{eq:1030}, and this change is essentially required for the proof of \eqref{eq:1030}. 

By the same argument as in \cite[(4.9)]{HW}, we have
\[
\tilde{Z}_r(\Phi_{\tf,r},s)=\frac{1}{2}\sum_\chi Z_r(\Phi_{\tf,r},s,\chi)
\]
where $\chi$ runs over all real characters on $\R_{>0}\Q^\times\bsl\A^\times$. 
Suppose that $\chi=\otimes_v\chi_v\notin \mathscr{X}(S_0)$. 
Then, we can take a prime $p\notin S_0$ such that $\chi_p$ is ramified and
\begin{equation}\label{eq:210513b}
\Phi_{\tf,r}(a_p x)=\Phi_{\tf,r}(x) \quad (\forall a_p\in\Z_p^\times), 
\end{equation}
holds.
Hence, we get $Z_r(\Phi_{\tf,r},s,\chi)\equiv 0$, and the proof is completed.
\end{proof}
\begin{remark}
The rational representation $\rho$ of $H_r$ on $V_r$ is faithful, but the representation $x\mapsto {}^t\!m x m$ of $\GL_r$ on $V_r$ is not. 
Hence, $Z_r(\Phi_{\tf,r},s,\chi)$ is suitable for Saito's explicit formula \cite{Saito2}, which we use in the proof of Theorem \ref{thm:asy}, but $\tilde{Z}_r(\Phi_{\tf,r},s)$ is not. 
This fact is also important for the study of global coefficients in the geometric side (cf. \cite{HW}). 
\end{remark}

Let $\psi$ be a non-trivial additive character on $\Q\bsl\A$, and a bilinear form $\langle \; , \; \rangle$ on $V_r(\A)$ is defined by $\langle x , y \rangle:=\Tr(x y)$. 
Let $\d x$ denote the self-dual measure on $V_r(\A)$ for $\psi(\langle \; , \; \rangle)$.
Then, a Fourier transform of $\Phi\in C^\inf(V_r(\A))$ is defined by
\[
\widehat{\Phi}(y)=\int_{V_r(\A)} \, \Phi(x)\, \psi(\langle x , y \rangle)\, \d x \qquad (y\in V_r(\A)). 
\]
For each $\Phi_0\in C^\inf_0(V_r(\A_f))$, we define its Fourier transform $\widehat{\Phi_0}$ in the same manner.
The zeta function $Z_r(\Phi_{\tf,r},s,\trep)$ satisfies the functional equation \cite{Shintani,Yukie}
\begin{equation}\label{eq:functional}
Z_r(\Phi_{\tf,r},s,\trep)=Z_r(\widehat{\Phi_{\tf,r}},\tfrac{r+1}{2}-s,\trep),
\end{equation}
where $\trep$ denotes the trivial representation on $\R_{>0}\Q^\times\bsl\A^\times$. 

Take a test function $\Phi_0\in C^\inf_0(V_r(\A_f))$ such that $\Phi_0({}^t\!k xk)=\Phi_0(x)$ holds for any $k\in \prod_{p<\inf}H_r(\Z_p)$ and $x\in V_r(\A_f)$, where $H_r(\Z_p)$ is identified with the projection of $\GL_1(\Z_p)\times \GL_r(\Z_p)$ into $H_r(\A_f)$.
We write $L_0$ for the subset of $V_r(\Q)$ which consists of the positive definite symmetric matrices contained in the support of $\Phi_0$.
It follows from the condition of $\Phi_0$ that $L_0$ is invariant for $\Gamma=H_r(\Z)=H_r(\Q)\cap H_r(\widehat{\Z})$. 
Put $\zeta_r(\Phi_0,s)=1$ for $r=0$.
For $r>0$, define a Shintani zeta function $\zeta_r(\Phi_0,s)$ as
\[
\zeta_r(\Phi_0,s)=\sum_{x\in L_0/\Gamma} \frac{\Phi_0(x)}{\#(\Gamma_x)\, \det(x)^s} .
\]
where $\Gamma_x=\{\gamma\in \Gamma\mid x\cdot \gamma=x\}$.
The zeta function $\zeta_r(\Phi_0,s)$ absolutely converges for $\Re(s)>\frac{r+1}{2}$, and is meromorphically continued to the whole $s$-plane (cf. \cite{Shintani}). Furthermore, $\zeta_r(\Phi_0,s)$ is holomorphic except for possible simple poles at $s=1,\frac{3}{2},\dots \frac{r+1}{2}$.
\begin{lem}\label{lem:functzeta}
Let $1\leq r\leq n$, $k_n>2n$, $h\in C_c^\inf(G(\A_f))$ and take a test function $\tf$ as $\tf=\tf_\uk h$. 
Then, there exists a rational function $\mathbf{C}_{n,r}(x_1,\dots,x_n)$ over $\R$ such that
\[
Z_r(\Phi_{\tf,r},n-\tfrac{r-1}{2},\trep)= \mathbf{C}_{n,r}(\uk) \times \zeta_r(\widehat{\Phi_{h,r}},r-n).
\]
\end{lem}
\begin{proof}
This can be proved by the functional equation \eqref{eq:functional} and the same argument as in \cite[Proof of Lemma 5.16]{Wakatsuki}.
\end{proof}
Note that $\zeta_r(\widehat{\Phi_{h,r}},s)$ is holomorphic in $\{s\in\C \mid \Re(s)\leq 0\}$, and $\mathbf{C}_{n,r}(x_1,\dots,x_n)$ is explicitly expressed by the Gamma function and the partitions, see \cite[(5.17) and Lemma 5.16]{Wakatsuki}. 
We will use this lemma for the regularization of the range of $\uk$. 
The zeta integral $Z_r(\Phi_{\tf,r},n-\tfrac{r-1}{2},\trep)$ was defined only for $k_n>2n$, but the right-hand side of the equality in Lemma \ref{lem:functzeta} is available for any $\uk$.
In addition, this lemma is necessary to estimate the growth of $I_\mathrm{unip}(f)$ with respect to $S=S_1\sqcup\{\inf\}$.
We later define a Dirichlet series $D_{m,u_S}^S(s)$ just before Proposition \ref{prop:1}, and the series $D_{m,u_S}^S(s)$ appears in the explicit formula of $Z_r(\Phi,s,\trep)$ when $r$ is even.
For the case that $r$ is even and $3<r<n$, it seems difficult to estimate the growth of its contribution to $Z_r(\Phi_{\tf,r},n-\tfrac{r-1}{2},\trep)$, but we can avoid such difficulty by this lemma, since the special value $\zeta_r(\widehat{\Phi_{h,r}},r-n)$ excludes the part of $D_{m,u_S}^S(s)$.

\begin{thm}\label{thm:tr}
Suppose $k_n>n+1$. Let $h_1\in \cH^\ur(G(\Q_{S_1}))^\kappa=\otimes_{p\in S_1}\cH^\ur(G(\Q_p))^\kappa$ and let $h$ be a test function on $G(\A_f)$ given as \eqref{eq:testft}.
Then there exists a positive constant $c_0$ such that, if $N\geq c_0 N_1^{2n\kappa}$, 
\begin{equation}\label{eq:zeta}
 \Tr\left( T_h|_{S_\uk(\Gamma(N))}\right)=\vol_G\, \vol(K(N))^{-1} \, h_1(1)\, d_\uk  + \frac{1}{2}\sum_{r=1}^n \mathbf{C}_{n,r}(\uk) \, \zeta_r(\widehat{\Phi_{h,r}},r-n).
\end{equation}
\end{thm}
\begin{proof}
Let $f=f_\uk h$ and $\tf=\tf_\uk h$. 
By Lemma \ref{lem:210513l1}, it is sufficient to prove that the geometric side $I_\mathrm{geom}(f)$ equals the right-hand side of \eqref{eq:zeta}.
If one uses the results in \cite{Arthur} and applies \cite[Lemma 8.4]{ST} by putting $\Xi:G\subset \GL_m$, $m=2n$, $B_\Xi=1$, $c_\Xi=c_0$ in their notations, then one gets $I_\mathrm{geom}(f)=I_\mathrm{unip}(f)$.
Hence, by Lemma \ref{lem:210513l2} and putting $h_{S_0}h^{S_0}=h$, we have
\begin{equation}\label{eq:210513a}
\Tr\left( T_h|_{S_\uk(\Gamma(N))}\right)=\vol_G\, \vol(K(N))^{-1} \,h_1(1)\, d_\uk  + \frac{1}{2}\sum_{r=1}^n \sum_{\chi\in \mathscr{X}(S_0)} Z_r(\Phi_{\tf,r},n-\tfrac{r-1}{2},\chi)
\end{equation}
for sufficiently large $k_n$.
Let $\mathscr{M}(a):=\diag(\overbrace{1,\dots,1}^n,\overbrace{a,\dots,a}^n)$ for $a\in\A^\times$. 
For any $a_p\in\Z_p^\times$, $b_p\in\Q_p^\times$, $\mu\in X_*(T)$, we have
\[
\mathscr{M}(a_p)^{-1}K_p(N)\mathscr{M}(a_p)=K_p(N), \quad \mathscr{M}(a_p)^{-1}\mu(b_p)\mathscr{M}(a_p)=\mu(b_p).
\]
Hence, \eqref{eq:210513b} holds for any $p<\inf$, and so $Z_r(\Phi_{\tf,r},n-\tfrac{r-1}{2},\chi)$ vanishes for any $\chi\neq \trep$.
Therefore, by Lemma \ref{lem:functzeta} we obtain the assertion \eqref{eq:zeta} for sufficiently large $k_n$.
By the same argument as in \cite[Proof of Theorem 5.17]{Wakatsuki}, we can prove that this equality \eqref{eq:zeta} holds in the range $k_n>n+1$, because the both sides of \eqref{eq:zeta} are rational functions of $\uk$ in that range, see Lemma \ref{lem:functzeta} and \cite[Proposition 5.3]{Wakatsuki}.
Thus, the proof is completed. 
\end{proof}

Let $S$ denote a finite subset of places of $\Q$, and suppose $\inf\in S$.
For each character $\chi=\otimes_v \chi_v$ on $\Q^\times\R_{>0}\bsl \A^\times$, we set
\[
L^S(s,\chi)=\prod_{p\notin S}L_p(s,\chi_p)  , \quad  L(s,\chi)=\prod_{p<\inf}L_p(s,\chi_p),
\]
\[
\zeta^S(s)=L^S(s,\trep)=\prod_{p\not\in S}(1-p^{-s})^{-1}  \quad \text{and} \quad
\zeta(s)=L(s,\trep),
\]
where $L_p(s,\chi_p)=(1-\chi_p(p)p^{-s})^{-1}$ if $\chi_p$ is unramified, and $L_p(s,\chi_p)=1$ if $\chi_p$ is ramified.
\begin{lem}\label{lem:zeta}
Let $s\in\R$. For $s>1$, one has
\[
\zeta^S(s) \leq \zeta(s) \quad   \text{and} \quad (\zeta^{S})'(s)\ll \frac{2s \zeta(s)}{s-1},
\]
where $(\zeta^{S})'(s)=\frac{\d}{\d s}\zeta^S(s)$.
For $s\leq -1$, one has
\[
|\zeta^S(s)| \leq (N_S)^{-s} |\zeta(s)|
\]
where $N_S=\prod_{p\in S\setminus \{\inf\}}p$.
\end{lem}
\begin{proof}
First of all, $(1-p^{-s})^{-1}\geq 1$ for $p\in S$. Hence $\zeta^S(s)\leq\zeta(s)$.
Let $\log \zeta^S(s) = \sum_{p\not\in S} \log(1-p^{-s})^{-1}$. Then
\[
\frac{(\zeta^{S})'(s)}{\zeta^S(s)}=\sum_{p\not\in S}\frac{-p^{-s}\log p}{1-p^{-s}}.
\]
If $s>1$, $1-p^{-s}\geq \frac{1}{2}$. Hence
\[
\left| \frac{(\zeta^{S})'(s)}{\zeta^S(s)} \right| \leq 2\sum_{p\not\in S} p^{-s}\log p \leq 2\sum_p p^{-s}\log p.
\]
By partial summation, $\sum_p p^{-s}\log p\leq \ds\int_1^\infty (\sum_{p\leq x}\log p) s x^{-s-1} \d x\leq \ds\int_1^\infty s x^{-s} \d x=\frac{s}{s-1}$.
Here we use the prime number theorem: $\sum_{p\leq x} \log p \sim x$.
Therefore, $(\zeta^{S})'(s)\ll \dfrac{2s \zeta(s)}{s-1}$.
\end{proof}

Set $\mathfrak{D}=\{ d(\Q^\times)^2 \mid d\in \Q^\times\}$.
For each $d\in \mathfrak{D}$, we denote by $\chi_d=\prod_v \chi_{d,v}$ the quadratic character on $\Q^\times\R_{>0}\bsl \A^\times$ corresponding to the quadratic field $\Q(\sqrt{d})$ via class field theory.
If $d=1$, then $\chi_d$ means the trivial character $\trep$.
For each positive even integer $m$, we set
\[
\varphi_{d,m}^S(s)=\zeta^S(2s-m+1)\zeta^S(2s) \frac{L^S(m/2,\chi_d)}{L^S(2s-m/2+1,\chi_d)} N(\mathfrak{f}_d^S)^{(m-1)/2-s}
\]
where $\mathfrak{f}_d^S$ denotes the conductor of $\chi_d^S=\prod_{p\not\in S}\chi_{d,p}$.
For each $u_S\in \Q_S=\prod_{v\in S}\Q_v$, one sets
\[
\mathfrak{D}(u_S)=\left\{ d(\Q^\times)^2 \mid d\in \Q^\times ,\,\, d\in u_S(\Q_S^\times)^2 \right\}.
\]
We need the Dirichlet series
\[
D_{m,u_S}^S(s)=\sum_{d(\Q^\times)^2  \in \mathfrak{D}(u_S) }\varphi_{d,m}^S(s).
\]
The following proposition is a generalization of \cite[Proposition 3.6]{IS2}.
\begin{prop}\label{prop:1}
Let $m\geq 2$ be an even integer.
Suppose $(-1)^{m/2}u_\inf>0$ for $u_S=(u_v)_{v\in S}$ (namely the term of $d (\Q^\times)^2= (\Q^\times)^2$ does not appear in $D_{m,u_S}^S(s)$ if $(-1)^{m/2}=-1$).
The Dirichlet series $D_{m,u_S}^S(s)$ is meromorphically continued to $\C$, and is holomorphic at any $s\in \Z_{\leq 0}$.
\end{prop}
\begin{proof}
See \cite[Corollary 4.23]{KTW} for the case $m>3$. 
For $m=2$, this statement can be proved by using \cite{HW,Yukie2}.
\end{proof}

\begin{thm}\label{thm:asy}
Fix a parameter $\uk$ such that $k_n>n+1$. 
Let $h_1\in \cH^\ur(G(\Q_{S_1}))^\kappa$ and let $h\in C_c^\inf (G(\A_f))$ be a test function on $G(\A_f)$ given as \eqref{eq:testft}. 
Suppose $\sup_{x\in G(\Q_{S_1})}|h_1(x)|\leq 1$.
Then, there exist positive constants $a$, $b$, and $c_0$ such that, if $N\geq c_0 N_1^{2n\kappa}$, 
\begin{equation*}
 \Tr\left( T_h|_{S_\uk(\Gamma(N))}\right)=\vol_G\, \vol(K(N))^{-1}\, h_1(1)\, d_{\uk} +\vol(K(N))^{-1}\, O(N_1^{a\kappa+b}N^{-n}).
\end{equation*}
Here the constants $a$ and $b$ do not depend on $\kappa$, $N_1$ and $N$. 
See Lemma \ref{lem:210513l2} for $\vol_G$ and $d_\uk$.
\end{thm}
\begin{proof}
Set
\[
I(\tf,r)=\vol(K(N)) \times \zeta_r(\widehat{\Phi_{h,r}},r-n)  \qquad  (1 \leq r \leq n).
\]
By Theorem \ref{thm:tr}, it is sufficient to prove $I(\tf,r)=O(N_1^{a\kappa+b}N^{-n})$.


Let $R$ be a finite set of places of $\Q$.
Take a Haar measure $\d x_\inf$ on $V_r(\R)$, and for each prime $p$ we write $\d x_p$ for the Haar measure on $V_r(\Q_p)$ normalized by $\int_{V_r(\Z_p)}\d x_p=1$.
For a test function $\Phi_R\in C_c^\inf(V_r(\Q_R))$ and an $H_r(\Q_R)$-orbit $\mathscr{O}_R\in V_r^0(\Q_{R})/H_r(\Q_{R})$, we set
\[
Z_{r,R}(\Phi_R,s,\mathscr{O}_R)=c_R \int_{\mathscr{O}_R} \Phi_R(x)\, |\det(x)|_{R}^{s-\frac{r+1}{2}}  \, \d x,
\]
where $c_R=\prod_{p\in R,\, p<\inf}(1-p^{-1})^{-1}$, $|\;|_R=\prod_{v\in R}|\;|_v$, and $\d x=\prod_{v\in R}\d x_v$.
It is known that $Z_{r,R}(\Phi_R,s,\mathscr{O}_R)$ absolutely converges for $\Re(s)\geq \frac{r+1}{2}$, and is meromorphically continued to the whole $s$-plane.

Suppose that $R$ does not contain $\inf$, that is, $R$ consists of primes.
Write $\eta_p(x)$ for the Clifford invariant of $x\in V_r^0(\Q_p)$, cf. \cite[Definition 2.1]{Ikeda}, and set $\eta_R((x_p)_{p\in R})=\prod_{p\in R}\eta_p(x_p)$.
For $\chi=\trep_R$ (trivial) or $\eta_R$, we put $(\Phi_R\chi)(x)=\Phi_R(x)\, \chi(x)$. 
It follows from the local functional equation \cite[Theorems 2.1 and 2.2]{Ikeda} over $\Q_p$ $(R=\{p\})$ that $Z_{r,p}(\Phi_p\chi,s,\mathscr{O}_p)$ is holomorphic in the range $\Re(s)<0$, and $Z_{r,p}(\Phi_p\chi,s,\mathscr{O}_p)$ possibly has a simple pole at $s=0$.  
Hence, for any $R$, $Z_{r,R}(\Phi_R\chi,s,\mathscr{O}_R)$ does not have any pole in the area $\Re(s)<0$, but it may have a pole at $s=0$.
Let $\widehat{\Phi_{R}}$ denote the Fourier transform of $\Phi_{R}\in C_c^\inf(V_r(\Q_{R}))$ over $\Q_{R}$ for $\prod_{v\in R}\psi_v(\langle \; , \; \rangle)$, where $\psi_v=\psi|_{\Q_v}$.

Define $\Phi_{h_1,r}(x)= h_1(\begin{pmatrix}I_n & * \\ O_n & I_n \end{pmatrix})\in C_c^\inf(V_r(\Q_{S_1}))$, $*=\begin{pmatrix}x&0 \\ 0 &0 \end{pmatrix}\in V_n$.
Note that this definition is compatible with $\Phi_{\tf,r}$ since $h_1$ is spherical for $\prod_{p\in S_1}K_p$.
Set
\[
\mathscr{Z}_r(S_1,h_1)=  \sum_{\mathscr{O}_{S_1}\in V_r^0(\Q_{S_1})/H_r(\Q_{S_1})} \Big|  Z_{r,S_1}(\widehat{\Phi_{h_1,r}}\chi_r,r-n,\mathscr{O}_{S_1}) \Big| ,
\]
where
\[
\chi_r=\begin{cases} \trep_{S_1} &  \text{if ($r$ is odd and $r<n$) or $r=2<n$}, \\ \eta_{S_1} & \text{if $r$ is even and $2<r<n$}, \end{cases}
\]
and
\[
\mathscr{Z}_n(S_1,h_1)= \sum_{\mathscr{O}_{S_1}\in V_r^0(\Q_{S_1})/H_r(\Q_{S_1})}\Big|  Z_{n,S_1}(\Phi_{h_1,n},\tfrac{n+1}{2},\mathscr{O}_{S_1}) \Big| \quad \text{if $r=n$.}
\]
It follows from Saito's formula \cite[Theorem 2.1 and \S 3]{Saito2} that the zeta function $\zeta_r(\widehat{\Phi_{h,r}},s)$ is expressed by a (finite or infinite) sum of  Euler products of $Z_{r,p}(\Phi_p\chi_p,s,\mathscr{O}_p)$ $(\chi_p=\trep_p$, $\eta_p)$ or its finite sums, and he explicitly calculated the local zeta function $Z_{r,p}(\Phi_p \chi_p,s,\mathscr{O}_p)$ in \cite[\S 2]{Saito} if $\Phi_p$ is the characteristic function of $V(\Z_p)$.
We shall prove $I(\tf,r)=O(N_1^{a\kappa+b}N^{-n})$ by using his results.

\vspace{2mm}
\noindent
(Case I) $r$ is odd and $r< n$. 
In the following, we set $S=S_1\sqcup\{\inf\}$. 
By Saito's formula we have
\begin{multline*}
I(\tf,r)=\text{(constant)}\times  N^{\frac{r(r-1)}{2} - r n} \times \sum_{\mathscr{O}_{S_1}\in V_r^0(\Q_{S_1})/H_r(\Q_{S_1})} Z_{r,S_1}(\widehat{\Phi_{h_1,r}},r-n,\mathscr{O}_{S_1}) \\
\times \zeta^S(\tfrac{r+1}{2}-n)\times \prod_{l=2}^n\zeta^S(l)^{-1}  \times \prod_{u=1}^{[r/2]}\zeta^S(2u) \zeta^S(2r-2n-2u+1).
\end{multline*}
Therefore, one has
\[
|I(\tf,r)| \ll N^{\frac{r(r-1)}{2} - r n}\times N_1^{2n^3} \times  \mathscr{Z}_r(S_1,h_1) 
\]
by using Lemma \ref{lem:zeta}.

\vspace{2mm}
\noindent
(Case II) $r$ is even and $3<r< n$. 
By Saito's formula, Proposition \ref{prop:1}, and Lemma \ref{lem:zeta}, one can prove that $|I(\tf,r)|$ is bounded by
\begin{multline*}
 N^{\frac{r(r-1)}{2} - r n} \times \mathscr{Z}_r(S_1,h_1) \times \Big| \zeta^S(\tfrac{r}{2})\times \prod_{l=2}^n\zeta^S(l)^{-1}  \times \prod_{u=1}^{r/2-1}\zeta^S(2u) \times  \prod_{u=1}^{r/2} \zeta^S(2r-2n-2u+1) \Big|  \\
\ll N^{\frac{r(r-1)}{2} - r n}\times N_1^{2n^3} \times  \mathscr{Z}_r(S_1,h_1)  
\end{multline*}
up to constant.
Note that Proposition \ref{prop:1} was used for this estimate, since it is necessary to prove the vanishing of the term including $D^S_{r,u_{S}}(s)$ in the explicit formula \cite[Theorem 3.3]{Saito2}.

\vspace{2mm}
\noindent
(Case III) $r=n$. In this case, we should use a method different from (Case I) and (Case II) since $Z_{r,S_1}(\widehat{\Phi_{h_1,r}}\chi,s,\mathscr{O}_{S_1})$ may have a simple pole at $s=r-n=0$.
Take an $n$-tuple $\underline{l}=(l_1,\dots,l_n)$ $(l_1\geq \cdots\geq l_n>2n)$, and let $n(x)=\begin{pmatrix}I_n& x \\ O_n & I_n \end{pmatrix}\in G$ where $x\in V_n$.
Recall that $\tf_{\underline{l}}$ satisfies the following two properties:
\begin{itemize}
\item[(i)] $\tf_{\underline{l}}(k^{-1}gk)=\tf_{\underline{l}}(g)$ $(\forall k\in K_\inf$, $g\in G(\R))$, see \cite[\S 5.3]{Wakatsuki}.
\item[(ii)] $\int_\R \tf_{\underline{l}}(g_1^{-1} n_1(t) g_2)\, \d t=0$ $(\forall g_1$, $g_2\in G(\R))$ where $n_1(t)=n((b_{ij})_{1\leq i,j\leq n})$, $b_{11}=t$ and $b_{ij}=0$ $(\forall (i,j)\neq(1,1))$, see \cite[Lemma 5.9]{Wakatsuki}.
\end{itemize}
By the property (i) we can define $\Phi_{\tf_{\underline{l}},n}(x)=\tf_{\underline{l}}(n(x))$ $(x\in V_n(\R))$.
\begin{lem}\label{lem:20210602}
For each orbit $\mathscr{O}_\inf\in V_n^0(\R) /H_n(\R)$, we have $Z_{n,\inf}(\Phi_{\tf_{\underline{l}},n},\frac{n+1}{2},\mathscr{O}_{\inf})=0$.
\end{lem}
\begin{proof}
Let $\mathscr{O}_\inf \neq I_n\cdot H_n(\R)$, and take a representative element $A$ of $\mathscr{O}_\inf$ as $A=\begin{pmatrix} 0& 0 & 1 \\ 0& \mathscr{A} &0 \\ 1&0&0 \end{pmatrix}$, $\mathscr{A}\in V_{n-2}^0(\R)$.
The orbit $\mathscr{O}_\inf$ is decomposed into $A\cdot \GL_n(\R) \sqcup (-A)\cdot \GL_n(\R)$.
The centralizer $H_{n(A)}$ of $n(A)$ in $H_n(\R)$ is given by $H_{n(A)}=\{  m(h)n(y) \mid h\in \O_A(n) $, $y\in V_n(\R) \}$, where $m(h)=\begin{pmatrix} {}^t h^{-1} & O_n \\ O_n & h \end{pmatrix}$ and $\O_A(n)=\{h\in \GL_n \mid {}^t h A h=A\}$.
Hence, by the property (ii), we have
\begin{multline*}
Z_{n,\inf}(\Phi_{\tf_{\underline{l}},n},\tfrac{n+1}{2},\mathscr{O}_{\inf})=\sum_{A'=\pm A}  \int_{\O_{A'}(n)\bsl \GL_n(\R)} \tf_{\underline{l}}(  m(h)^{-1} n(A') m(h)  ) \, |\det(h)|^{n+1} \, d h \\
= \sum_{A'=\pm A}  \int_{\mathscr{N}\O_{A'}(n)\bsl \GL_n(\R)} \int_\R \tf_{\underline{l}}(  m(h)^{-1} n(A') n_1(2t) m(h)  ) \, |\det(h)|^{n+1}\, \d t \, d h=0,
\end{multline*}
where $\mathscr{N}=\{ (b_{ij}) \mid b_{jj}=1$ $(1\leq j\leq n)$, $b_{n1}\in \R$, and $b_{ij}=0$ otherwise$\}$.

In the case $s=\frac{n+1}{2}$, we note that $|\det(x)|$ vanishes in the integral of $Z_{n,\inf}(\Phi_{\tf_{\underline{l}},n},\tfrac{n+1}{2},\mathscr{O}_{\inf})$.
Hence, it follows from the property (ii) that
\[
\sum_{\mathscr{O}_\inf\in V_n^0(\R)/H_n(\R)} Z_{n,\inf}(\Phi_{\tf_{\underline{l}},n},\tfrac{n+1}{2},\mathscr{O}_\inf)= \int_{V_n(\R)}\Phi_{\tf_{\underline{l}},n}(x) \, \d x=0,
\]
and so we also find $Z_{n,\inf}(\Phi_{\tf_{\underline{l}},n},\frac{n+1}{2}, I_n\cdot H_n(\R))=0$.
\end{proof}
By Lemmas \ref{lem:functzeta} and \ref{lem:20210602}, the residue formula \cite[Ch. 4]{Yukie} of $Z_n(\Phi,s,\trep)$ and the same argument as in \cite[Proof of Theorem 4.22]{HW} we obtain
\begin{multline*}
\zeta_r(\widehat{\Phi_{h,r}},0)=\mathbf{C}_{n,n}(\underline{l})^{-1}\, Z_n(\Phi_{\tf_{\underline{l}}h,r} ,\tfrac{n+1}{2},\trep)= \mathbf{C}_{n,n}(\underline{l})^{-1}\,  \vol(H_n(\Q)\bsl H_n(\A)^1) \\
\times \int_{V(\R)}\Phi_{\tf_{\underline{l}},n}(x_\inf) \, \log|\det(x_\inf)|_\inf \, \d x_\inf \, \int_{V(\Q_{S_1})} \Phi_{h_1,n}(x_{S_1}) \, \d x_{S_1} \, N^{-\frac{n(n+1)}{2}},
\end{multline*}
where $H_n(\A)=\{(a,m)\in H_n(\A)\mid |a^n\det(m)^2|=1\}$.
From this we have $|I(\tf,r)| \ll N^{-\frac{n(n+1)}{2}} \times  \mathscr{Z}_r(S_1,h_1)$.

\vspace{2mm}
\noindent
(Case IV) $r=2<n$. 
By Saito's formula \cite[Theorem 4.15]{HW} we have
\[
|I(\tf,r)| \ll N^{1 - 2 n}  \times \mathscr{Z}_2(S_1,h_1) \times \Big| \zeta^S(2)^{-1}  \zeta^S(3-2n)\Big|    \times \max_{u_S\in \Q_S^\times/(\Q_S^\times)^2,\; u_\inf<0} |D_{2,u_S}^S(2-n)| .
\]
Hence, it is enough to give an upper bound of $|D_{2,u_S}^S(2-n)|$ for $u_\inf<0$.
Choose a representative element $u_S=(u_v)_{v\in S}$ satisfying $u_p\in \Z_p$ $(p\in S_1)$.
Take a test function $\Phi=\otimes_v \Phi_v$ such that the support of $\Phi_\inf$ is contained in $\{x\in V_2^0(\R)\mid \det(x)>0\}$ and $\Phi_p$ is the characteristic function of $\diag(1,-u_p)+p^2 V_2(\Z_p)$ (resp. $V_2(\Z_p)$) for each $p\in S_1$ (resp. $p\not\in S$).
Let
\[
\Psi(y,yu)=\int_{K_2} \widehat\Phi({}^t\!k \begin{pmatrix} 0&y \\ y&yu\end{pmatrix}k)\, \d k  ,\qquad  K_2=\O(2,\R)\times \prod_p \GL_2(\Z_p),
\]
and we set
\[
T(\Phi,s)=\frac{\d}{\d s_1}T(\Phi,s,s_1)\Big|_{s_1=0}, \quad T(\Phi,s,s_1)=\int_{\A^\times}\int_\A |y^2|^s \|(1,u)\|^{s_1}\Psi(y,yu) \, \d u \d^\times y .
\]
By \cite[Lemma 1]{Shintani}, one obtains $Z_{2,\inf}(\widehat{\Phi_\inf},n-\tfrac{1}{2},\mathscr{O}_\inf)= 0$ for any orbit $\mathscr{O}_\inf$ in $V_2^0(\R)$.
Therefore, from the functional equation \cite[Corollary (4.3)]{Yukie2} one deduces
\[
|N_1^{-6} D_{2,u_S}^S(2-n)| \ll |Z_{2,S}(\Phi_S,2-n,\mathscr{O}_S) \, D_{2,u_S}^S(2-n)|  = \Big| 2^{-1}T(\widehat\Phi,n-\tfrac{1}{2}) \Big|.
\]
By \cite[Proposition (2.12) (2)]{Yukie2}, one gets
\[
|T(\widehat{\Phi},n-\tfrac{1}{2})|\ll N_1^{4n-2}\times \left\{ \zeta^S(2n-2) + |  (\zeta^{S})'(2n-2) |+ \Big|  \frac{(\zeta^{S})'(2n-1) \zeta^S(2n-2) }{\zeta^S(2n-1)}  \Big|   \right\}  ,
\]
where $(\zeta^{S})'(s)=\frac{\d}{\d s}\zeta^S(s)$, because $\mathrm{Supp}(\widehat\Phi_p)\subset p^{-2}V(\Z_p)$ for any $p\in S_1$.
Therefore, one gets $|D_{2,u_S}^S(2-n)|\ll N_1^{4n+4}$ by Lemma \ref{lem:zeta}.

\vspace{2mm}
The final task is to prove $\mathscr{Z}_r(S_1,h_1) \ll N_1^{a\kappa+b}$ for some $a$ and $b$.
Using the local functional equations in \cite[Theorem 2.1]{Ikeda} (see also \cite{Sweet}), one gets
\[
\mathscr{Z}_r(S_1,h_1) \ll N_1^{c} \times \sum_{\mathscr{O}_{S_1}\in V_r^0(\Q_{S_1})/H_r(\Q_{S_1})}Z_{r,S_1}(|\Phi_{h_1,r}|,n-\tfrac{r-1}{2},\mathscr{O}_{S_1})
\]
for some $c\in \N$. 
By \cite[Lemma 2.1.1]{Assem1} and the assumption $\sup_{x\in G(\Q_{S_1})}|h_1(x)|\leq 1$, we have
\[
|\Phi_{h_1,r}| \leq \Phi_{S_1,r,-\kappa}
\]
where $\Phi_{S_1,r,-\kappa}$ denotes the characteristic function of $\otimes_{p\in S_1}p^{-\kappa} V_r(\Z_p)$. 
Hence, by change of variables, we get
\begin{multline*}
Z_{r,S_1}(|\Phi_{h_1,r}|,n-\tfrac{r-1}{2},\mathscr{O}_{S_1}) \leq Z_{r,S_1}(\Phi_{S_1,r,-\kappa},n-\tfrac{r-1}{2},\mathscr{O}_{S_1})\\
= N_1^{\kappa nr -\frac{\kappa r(r-1)}{2}} Z_{r,S_1}(\Phi_{S_1,r,0},n-\tfrac{r-1}{2},\mathscr{O}_{S_1}) \le N_1^{\kappa nr -\frac{\kappa r(r-1)}{2}}.
\end{multline*}
It follows from classification theory of quadratic forms that $\#(V_r^0(\Q_{S_1})/H_r(\Q_{S_1})) \ll N_1$. 
Therefore, we obtain a desired upper bound for $\mathscr{Z}_r(S_1,h_1)$. 
Thus, we obtain $I(\tf,r)=O(N_1^{a\kappa+b}N^{-n})$.
\end{proof}

\begin{remark}\label{STD}
We give some remarks on Shin-Templier's work \cite{ST} and Dalal's work \cite{Dalal}. 
In the setting of \cite{ST}, they considered ``all" cohomological representations 
as a family which exhausts an $L$-packet at infinity since they  
chose the Euler-Poincar\'e pseudo-coefficient at the infinite place. 
Then there is no contribution from non-trivial unipotent conjugacy classes. 
Therefore, our work is different from Shin-Templier's work in that we can consider only holomorphic forms in an $L$-packet.

Shin suggested to consider a family of automorphic representations whose infinite type is any fixed discrete series representation. Dalal \cite{Dalal} carried it out in the weight aspect by using the stable trace formula. 
The stabilization allows us to remove the contribution $I_3(f)$ (see the introduction), but instead of $I_3(f)$, the contributions of endoscopic groups have to enter. 
Dalal obtained a good bound for them by using the concept of hyperendoscopy introduced by Ferrari \cite{Ferrari}. 
In studying the level aspect, it seems difficult to get a sufficient bound for the growth of concerning hyperendoscopic groups, but our careful analysis (cf. the proof of Theorem \ref{thm:asy}) shows that estimating unipotent contributions is simpler than using stable trace formula. 
In fact, some of zeta integrals $Z_r(\Phi_{\tf,r},s,\chi)$ probably correspond to the contributions of endoscopic groups of $\Sp(2n)$, and so we still need similar computations even if we use the stable trace formula. 
\end{remark}



\section{Arthur classification of Siegel modular forms}\label{Arthur}

In this section we study Siegel modular forms in terms of Arthur's classification 
\cite{Ar-book} (see Section 1.4 and 1.5 of loc.cit.). 
Recall $G=\Sp(2n)/\Q$. We call a Siegel cusp form which comes from smaller  
groups by Langlands functoriality ``a non-genuine form." 
In this section, we estimate the dimension of the space of non-genuine forms and 
show that they are negligible.
    
Let $F\in HE_\uk(N)$ (see Section \ref{sec:pre}), and $\pi=\pi_F$ be 
the corresponding automorphic representation of $G(\A)$. According to Arthur's classification, $\pi$ can be described by using the global 
$A$-packets. 
Let us recall some notations. A (discrete) global $A$-parameter is a symbol   
$$\psi=\pi_1[d_1]\boxplus \cdots\boxplus \pi_r[d_r]
$$   
satisfying the following conditions:
\begin{enumerate}
\item for each $i$ $(1\le i \le r)$, $\pi_i$ is an irreducible unitary cuspidal self-dual automorphic representation  of ${\rm GL}_{m_i}(\A)$. 
In particular, the central character $\omega_i$ of $\pi_i$ is trivial or quadratic; 
\item for each $i$, $d_i\in \Z_{>0}$ and $\ds\sum_{i=1}^r m_id_i=2n+1$;
\item if $d_i$ is odd, $\pi_i$ is orthogonal, i.e., $L(s,\pi_i,{\rm Sym}^2)$ has a pole at $s=1$;
\item if $d_i$ is even, $\pi_i$ is symplectic, i.e., $L(s,\pi_i,\wedge^2)$ has a pole at $s=1$;
\item $\omega^{d_1}_1\cdots \omega^{d_r}_r=\trep$; 
\item if $i\neq j$, $\pi_i\simeq \pi_j$, then $d_i\neq d_j$. 
\end{enumerate}

We say that two global $A$-parameters $\boxplus_{i=1}^r\pi_i[d_i]$ and $\boxplus_{i=1}^{r'}\pi'_i[d'_i]$ are 
equivalent if $r=r'$ and there exists $\sigma\in \mathfrak S_r$ such that $d'_i=d_{\sigma(i)}$ and 
$\pi'_i=\pi_{\sigma(i)}$. 
Let $\Psi(G)$ be the set of equivalent classes of global $A$-parameters. 
For each $\psi \in \Psi(G)$, one can associate a set $\Pi_\psi$ of equivalent classes of simple admissible $G(\A_f)\times (\frak g,K_\infty)$-modules (see \cite{Ar-book}). 
The set $\Pi_\psi$ is called a global $A$-packet for $\psi$.
 
\begin{Def}\label{cap-endo}Let $\psi=\ds\boxplus_{i=1}^r\pi_i[d_i]$ be a global $A$-parameter. 
\begin{itemize}
\item $\psi$ is said to be semi-simple if $d_1=\cdots=d_r=1$; otherwise, $\psi$ is said to be non-semi-simple; 
\item $\psi$ is said to be simple if $r=1$ and $d_1=1$.   
\end{itemize}
\end{Def}

By \cite[Theorem 1.5.2]{Ar-book} (though our formulation is slightly different from the 
original one), we have a following decomposition  
\begin{equation}\label{AD}
L^2_{{\rm disc}}(G(\Q)\bs G(\A))\simeq \bigoplus_{\psi\in \Psi(G)}\bigoplus_{\pi\in \Pi_\psi}m_{\pi,\psi}\pi
\end{equation}
where $m_{\pi,\psi}\in \{0,1\}$ (cf. see \cite[Theorem 2.2]{Atobe} for $m_{\pi,\psi}$). 
We have an immediate consequence of (\ref{AD}):

\begin{prop}\label{first-est} Let $1_{K(N)}$ be the characteristic function of $K(N)\subset G(\A_f)$. Then 
$$S_{\uk}(\Gamma(N))=
\bigoplus_{\psi\in \Psi(G)}\bigoplus_{\pi\in \Pi_\psi\atop \pi_\infty\simeq \sigma_\uk}m_{\pi,\psi}\pi_f^{K(N)}
$$ and 
\begin{equation}\label{characteristic}
|HE_{\uk}(N)|={\rm vol}(K(N))^{-1}\sum_{\psi\in \Psi(G)}
\sum_{\pi\in \Pi_\psi \atop \pi_\infty\simeq \sigma_\uk}m_{\pi,\psi}{\rm tr}(\pi_f(1_{K(N)})).
\end{equation}
\end{prop}

\begin{thm}\label{Ram} Assume (\ref{suff-reg}). 
For a global $A$-parameter $\psi=\ds\boxplus_{i=1}^r\pi_i[d_i]$, suppose that there exists $\pi\in \Pi_\psi$ with $\pi_\infty\simeq \sigma_\uk$. Then $\psi$ is semi-simple, i.e., $d_i=1$ for all $i$, and each $\pi_i$ is regular algebraic and satisfies the Ramanujan conjecture, i.e., $\pi_{i,p}$ 
is tempered for any $p$.
\end{thm}
\begin{proof} By the proof of \cite[Corollary 8.5.4]{CL}, we see that $d_1=\cdots=d_r=1$. Hence $\psi$ is semi-simple.
Further, by comparing infinitesimal characters $c(\pi_\infty),\ c(\psi_\infty)$ of $\pi_\infty,\ \psi_\infty$ 
respectively, we see that each $\pi_i$ is regular algebraic by \cite[Corollary 6.3.6, p.164 and Proposition 8.2.10, p.196]{CL}. 
It follows from \cite{Ca1} and \cite{Ca2} that $\pi_{i,p}$ is tempered for any $p$.
\end{proof}

Therefore, for each finite prime $p$, 
the local Langlands parameter at $p$ of $\pi$ is described as one of 
the isobaric sum $\boxplus_{i=1}^r\pi_{i,p}$ which is an admissible representation of ${\rm GL}_{2n+1}(\Bbb Q_p)$. 

\begin{Def} 
We denote by $HE_{\uk}(N)^{ng}$ the subset of $HE_{\uk}(N)$ consisting of all 
forms which belong to  
$$\bigoplus_{\psi\in \Psi(G)\atop \psi\text{:non-simple}}
\bigoplus_{\pi\in \Pi_\psi \atop \pi_\infty\simeq \sigma_\uk}m_{\pi,\psi}\pi_f^{K(N)},
$$ 
under the isomorphism $(\ref{AD})$. A form in this space is called a non-genuine form. 

Similarly, we denote by $HE_{\uk}(N)^{g}$ the subset of $HE_{\uk}(N)$ consisting of all forms which belong to  
$$\bigoplus_{\psi\in \Psi(G)\atop \psi\text{:simple}}
\bigoplus_{\pi\in \Pi_\psi \atop \pi_\infty\simeq \sigma_\uk}m_{\pi,\psi}\pi_f^{K(N)},
$$ 
under the isomorphism $(\ref{AD})$. A form in this space is called a genuine form. 
\end{Def}
 
\begin{Def}\label{gln} Denote by $\Pi(\GL_n(\Bbb R))^c$ the isomorphism classes of  all irreducible cohomological admissible $(\mathfrak{gl}_n,O(n))$-modules.
For $\tau_\infty\in \Pi(\GL_n(\Bbb R))^c$ and a quasi-character 
$\chi:\Q^\times\bs\A^\times\lra \C^\times$,  
we define 
$$L^{{\rm cusp},{\rm ort}}({\rm GL}_n(\Q)\bs {\rm GL}_n(\A),\tau_\infty,\chi):=
\bigoplus_{\pi:\text{orthogonal}\atop \pi_\infty\simeq \tau_\infty, \omega_\pi=\chi}
m(\pi)\pi$$ 
and $$L^{{\rm cusp},{\rm ort}}(K^{{\rm GL}_n}(N),\tau_\infty,\chi):=
\bigoplus_{\pi:\text{orthogonal}\atop \pi_\infty\simeq \tau_\infty, \omega_\pi=\chi}
m(\pi)\pi^{K^{{\rm GL}_n}(N)}$$
where  the direct sums are taken over the isomorphism classes of all orthogonal cuspidal automorphic representations of ${\rm GL}_n(\A)$ 
and $\omega_\pi$ stands for the central character of $\pi$. 
The constant $m(\pi)$ is the multiplicity of $\pi$ in 
$L^2({\rm GL}_n(\Q)\bs {\rm GL}_n(\A))$ which satisfies 
$m(\pi)\in\{0,1\}$ by \cite{Shalika}.  
Here 
 $K^{{\rm GL}_{n}}(N)$ is the principal congruence subgroup of ${\rm GL}_{n}(\widehat{\Z})$ of level $N$.  
Put 
$$l^{{\rm cusp},{\rm ort}}(n,N,\tau_\infty,\chi):={\rm dim}_\C(L^{{\rm cusp},{\rm ort}}(K^{{\rm GL}_n}(N),\tau_\infty,\chi))
$$ 
for simplicity. Clearly, $l^{{\rm cusp},{\rm ort}}(1,N,\tau_\infty,\chi)=
|\widehat{\Z}^\times/(1+N\widehat{\Z})^\times|=\varphi(N)$ where $\varphi$ stands for 
Euler's totient function.   
\end{Def}

Let $P(2n+1)$ be the set of all partitions of $2n+1$ and 
$P_{\underline{m}}$ be the standard parabolic subgroup of ${\rm GL}_{2n+1}$ 
associated to a partition $2n+1=m_1+\cdots+m_r$, and $\underline{m}=(m_1,...,m_r)$.

In order to apply the formula (\ref{characteristic}), it is necessary to study the transfer of Hecke elements in 
the local Langlands correspondence established by \cite[Theorem 1.5.1]{Ar-book}. 
We regard $G={\rm Sp}(2n)$ as a twisted elliptic endoscopic subgroup of ${\rm GL}_{2n+1}$ (cf. \cite{GV} or \cite{Oi}). 

\begin{prop}\label{transfer} Let $N$ be an odd positive integer.
Put $S_N:=\{p\, \text{prime}\ :\ p\mid N\}$.
For the pair $({\rm GL}_{2n+1},G)$, the characteristic function of 
${\rm vol}(K(N))^{-1}1_{K(N)}$ as an element of 
$C^\infty_c(G(\Q_{S_N}))$ is transferred to  
$${\rm vol}(K^{{\rm GL}_{2n+1}}(N))^{-1}1_{K^{{\rm GL}_{2n+1}}(N)}$$  
as an element of $C^\infty_c({\rm GL}_{2n+1}(\Q_{S_N}))$. 
\end{prop} 
\begin{proof}It follows from \cite[Lemma 8.2.1-(i)]{GV}.  
\end{proof}

Applying Proposition \ref{transfer}, we have the following.  

\begin{prop}\label{second-est} Assume (\ref{suff-reg}) and $N$ is odd. 
Then  
$|HE_{\uk}(N)^{ng}|$ is bounded by 
$$
\frac{A_n(N)}{\varphi(N)}\sum_{\underline{m}=(m_1,\ldots,m_r)\in P(2n+1)\atop r\geq 2}
\sum_{ \tau_i\in\Pi({\rm GL}_{m_i}(\R))^c\atop c(\boxplus_{i=1}^r \tau_i)=c(\sigma_\uk)}
\sum_{\chi_i:\Q^\times\bs\A^\times \lra\C^\times\atop \chi^2_i=1,\ c(\chi)|N}d_{P_{\underline{m}}}(N)\prod_{i=1}^r
l^{{\rm cusp},{\rm ort}}(m_i,N,\tau_{i},\chi_i),
$$
where the second sum is indexed by all $r$-tuples $(\tau_1,\ldots,\tau_r)$ such that  
$\tau_i\in\Pi({\rm GL}_{m_i}(\R))^c$  and 
$c(\boxplus_{i=1}^r \tau_i)=c(\sigma_{\uk})$ (the equality of the infinitesimal characters). Further $c(\chi)$ stands for the conductor of $\chi$ and 
$\varphi(N)=|(\Z/N\Z)^\times|$. 
Here
\begin{enumerate}
\item $A_n(N):= 2^{(2n+1)\omega(N)}$ where $\omega(N):=|\{p\ prime\ :\ p|N\}|$;
\item $d_{P_{\underline{m}}}(N)=|P_{\underline{m}}(\Z/N\Z)\bs {\rm GL}_{2n+1}(\Z/N\Z)|={\rm vol}({K^{{\rm GL}_{2n+1}}(N)})^{-1}/|P_{\underline{m}}(\Z/N\Z)|$.
\end{enumerate}
\end{prop}
\begin{proof} Let $\pi=\pi_\infty\otimes\otimes'_p \pi_p$ be an element of $\Pi_\psi$ for 
$\psi=\boxplus_{i=1}^r\pi_i$. 
Let $\Pi_p$ be the local Langlands correspondence of $\pi_p$ to ${\rm GL}_{2n+1}(\Q_p)$ 
established by \cite[Theorem 1.5.1]{Ar-book} and $\mathcal{L}(\Pi_p):L_{\Q_p}\lra {\rm GL}_{2n+1}(\C)$ be the local $L$-parameter of $\Pi_p$, where $L_{\Q_p}=W_{\Q_p}$ for each 
$p<\infty$ and $L_{\R}=W_{\R}\times {\rm SL}_2(\C)$. 
Since the localization $\psi_p$ of the global $A$-parameter $\psi$ at $p$ is 
tempered by Theorem \ref{Ram}, we see that $\mathcal{L}(\Pi_p)$ is equivalent to $\psi_p$. 
Since $\mathcal{L}(\Pi_p)$ is independent of $\pi\in \Pi_\psi$ and multiplicity one for ${\rm GL}_{2n+1}(\A)$ holds, 
the isobaric sum $\psi=\boxplus_{i=1}^r \pi_i$ as an automorphic representation of ${\rm GL}_{2n+1}(\A)$ gives rise to a unique  
global L-parameter on $\Pi_\psi$. On the other hand, it follows from \cite[Theorem 1.5.1]{Ar-book} that 
$|\Pi_{\psi_p}|\le 2^{2n+1}$ for the local A-packet $\Pi_{\psi_p}$ at $p$ if $p|N$, and $\Pi_{\psi_p}$ is a singleton if $p\nmid N$. 
It yields that $|\Pi_\psi|\le 2^{(2n+1)\omega(N)}$. 
Since the local Langlands correspondence $\pi_p\mapsto \Pi_p$ satisfies the character relation 
by \cite[Theorem 1.5.1]{Ar-book}, it follows from Proposition \ref{transfer} 
that for each $\pi\in \Pi_\psi$, 
\begin{eqnarray*}
&&{\rm dim}(\pi^{K(N)}_f)={\rm vol}(K(N))^{-1} {\rm tr}(\pi(1_{K(N)}))=
{\rm vol}(K^{{\rm GL}_{2n+1}}(N))^{-1}{\rm tr}((\boxplus_{i=1}^r\pi_i)(1_{K^{{\rm GL}_{2n+1}}(N)}))\\
&&\phantom{xxxxxxxxx} ={\rm dim}((\boxplus_{i=1}^r\pi_i)_f^{{K^{{\rm GL}_{2n+1}}(N)}}),
\end{eqnarray*}
where we denote by $\pi_f=\otimes_{p<\infty}' \pi_p$ the finite part of the cuspidal representation $\pi$. Plugging this into Proposition \ref{first-est}, we have 
\begin{eqnarray}
|HE_{\uk}(N)^{ng}|&=&{\rm vol}(K(N))^{-1}
\sum_{\psi=\boxplus_{i=1}^r\pi_i\in \Psi(G),\ r\ge 2\atop c(\psi_\infty)=c(\sigma_\uk)} \sum_{\pi\in \Pi_\psi}m_{\pi,\psi}
{\rm tr}(\pi_f(1_{K(N)})) \nonumber \\
&\le &\frac{A_n(N)}{\varphi(N)}
\sum_{\psi=\boxplus_{i=1}^r\pi_i\in \Psi(G),\ r\ge 2\atop c(\psi_\infty)=c(\sigma_\uk)}
{\rm dim}((\boxplus_{i=1}^r\pi_i)_f^{{K^{{\rm GL}_{2n+1}}(N)}}).  \nonumber 
\end{eqnarray}  
where 
$\ds\frac{1}{\varphi(N)}$ is inserted because of the condition on the central characters in global $A$-parameters. Here 
$r\ge 2$ is essential to gain the factor $\ds\frac{1}{\varphi(N)}$ 
(see Remark \ref{characters-exp}).

Next we describe ${\rm dim}((\boxplus_{i=1}^r\pi_i)_f^{{K^{{\rm GL}_{2n+1}(N)}}})$ in terms of 
the data $(m_i,N,\tau_i,\chi_i)$ with $1\le i\le r$. 
Since 
$$P_{\underline{m}}(\A_f)\bs {\rm GL}_{2n+1}(\A_f)/K(N)\simeq 
P_{\underline{m}}(\widehat{\Z})\bs {\rm GL}_{2n+1}(\widehat{\Z})/K(N) \simeq P_{\underline{m}}(\Z/N\Z)\bs {\rm GL}_{2n+1}(\Z/N\Z)
$$
and a complete system of the representatives can be taken 
from elements in ${\rm GL}_{2n+1}(\widehat{\Z})$ and therefore they normalize $K(N)$. 
Then a standard method for fixed vectors of 
an induced representation shows that 
$${\rm dim}((\boxplus_{i=1}^r\pi_i)_f^{{K^{{\rm GL}_{2n+1}}(N)}})
=d_{P_{\underline{m}}}(N) \prod_{i=1}^r
{\rm dim}(\pi^{{K^{{\rm GL}_{m_i}(N)}}}_{i,f}),
$$ 
Here if $\chi_i$ is the central character of $\pi_i$ and $\pi_{i,\infty}\simeq \tau_i$, then ${\rm dim}(\pi^{{K^{{\rm GL}_{m_i}(N)}}}_{i,f})=l^{{\rm cusp},{\rm ort}}(m_i,N,\tau_{i},\chi_i)$. Notice that the conductor of $\chi_i$ is a divisor of $N$. Summing up, we have the claim. 
\end{proof}
\begin{remark}\label{characters-exp}Let $r\ge 2$. The group homomorphism $((\Z/N\Z)^\times)^r\lra (\Z/N\Z)^\times,\ 
(x_1,\ldots,x_r)\mapsto x_1\cdots x_r$ is obviously surjective and it yields  
$$\left|\{(\chi_1,\ldots,\chi_r)\in \widehat{(\Z/N\Z)^\times}^r\ |\ \chi_1\cdots\chi_r=1\}\right|
=\frac{\left|\widehat{(\Z/N\Z)^\times}^r\right|}{\varphi(N)}.$$
This trivial equality explains the appearance of the factor $\ds\frac{1}{\varphi(N)}$ in 
Proposition \ref{second-est}. 
\end{remark}

Next we study $l^{{\rm cusp},{\rm ort}}(n,N,\tau,\chi)$ for $\tau\in \Pi({\rm GL}_n(\R))^c$ 
and for $n\ge 2$. Now if $\pi$ is a cuspidal representation of $\GL_{2m+1}$ which is orthogonal, i.e., $L(s,\pi, {\rm Sym}^2)$ has a pole at $s=1$, then $\pi$ comes from a cuspidal representation $\tau$ on $\Sp(2m)$. In this case, the central character $\omega_{\pi}$ of $\pi$ is trivial.

If $\pi$ is a cuspidal representation of $\GL_{2m}$ which is orthogonal, i.e., $L(s,\pi, {\rm Sym}^2)$ has a pole at $s=1$, then 
$\omega_\pi^2=1$; If $\omega_\pi=1$, $\pi$ comes from a cuspidal representation $\tau$ on the split orthogonal group $\SO(m,m)$;
If $\omega_\pi\ne 1$, then $\pi$ comes from a cuspidal representation $\tau$ on the quasi-split orthogonal group $\SO(m+1,m-1)$.

First we consider the case when $\chi$ is trivial in estimating 
$l^{{\rm cusp},{\rm ort}}(2n+\delta, N,\tau,\chi)$ where $\delta=0$ or 1. 
For a positive integer $n$, let $H=\begin{cases} \SO(n,n), &\text{if $G'=\GL_{2n}$}\\
\Sp(2n), &\text{if $G'=\GL_{2n+1}$}\end{cases}.$

We regard $H$ as a twisted elliptic endoscopic subgroup $G'$. 

\begin{prop}\label{transfer-ort} Let $N$ be an odd 
positive integer. 
For the pair $(G',H)$, the characteristic function of 
${\rm vol}(K^H(N))^{-1}1_{K^H(N)}$ as an element of 
$C^\infty_c(H(\Q_{S_N}))$ is transferred to  
$${\rm vol}(K^{G'}(N))^{-1} 1_{K^{G'}(N)}$$
as an element of $C^\infty_c(G'(\Q_{S_N}))$. 
\end{prop} 
\begin{proof} It follows from \cite[Lemma 8.2.1-(i)]{GV}.  
\end{proof}

Each cuspidal representation $\pi$ of $G'(\A)$ contributing to $l^{{\rm cusp},{\rm ort}}(N,\tau,\trep)$ can be 
regarded as a simple $A$-parameter. Also 
as a cuspidal representation, it  
strongly descends to a generic cuspidal representation $\Pi_\pi$ of $H(\A)$ 
whose $L$-parameter 
$\mathcal{L}(\Pi_{\tau})$ at infinity of 
$\Pi_{\pi}$ is same as one of $\pi_\infty$.   
In this setting, by \cite[Proposition 8.3.2-(b), p.483]{Ar-book}, the problem is reduced to estimate 
$$L^{{\rm cusp,gen}}(H,N,\mathcal{L}(\Pi_\tau),\trep)
:=\bigoplus_{\pi\subset L^{{\rm cusp,generic},{\rm ort}}(H(\Q)\bs H(\A),\mathcal{L}(\Pi_{\tau}),\trep)}
m(\pi)\pi^{K^{H}(N)},\, m(\pi)\in\{0,1,2\}
$$ 
where $\pi$ runs over all irreducible unitary, cohomological orthogonal cuspidal automorphic representations of $H(\A)$ whose 
$L$-parameter at infinity is isomorphic to $\mathcal{L}(\Pi_{\tau})$ with the 
central character $\chi=\trep$.  

\begin{prop}\label{trivial-case} Keep the notations as above. Then
\begin{itemize}
\item $l^{{\rm cusp},{\rm ort}}(2n+\delta,N,\tau,\trep)\le C_n(N){\rm dim}(L^{{\rm cusp,gen}}(H,N,\mathcal{L}(\Pi_\tau),\trep))$, where 
$C_n(N):=2^{(2n+\delta)\omega(N)}$ 
and $\delta=\begin{cases} 0, &\text{if $G'=\GL_{2n}$}\\1, &\text{if $G'=\GL_{2n+1}$}\end{cases}$.
\item ${\rm dim}(L^{{\rm cusp,gen}}(H,N,\mathcal{L}(\Pi_\tau),\trep)) \ll c\cdot {\rm vol}(K^H(N))^{-1}\sim c N^{{\rm dim}(H)}$ for some $c>0$, when the infinitesimal character of $\mathcal{L}(\Pi_\tau)$ is fixed and $N\to\infty$. 
\end{itemize}
\end{prop}
\begin{proof} The first claim follows from \cite[Proposition 8.3.2-(b), p.483]{Ar-book} with 
a completely similar argument of Proposition \ref{second-est}. 
 
The second claim follows from \cite{Savin}.
\end{proof}

Next we consider the case when $\chi$ is a quadratic character. 
In this case, a cuspidal representation $\pi$  
contributing to $L^{{\rm cusp},{\rm ort}}(K^{{\rm GL}_n}(N),\tau_\infty,\chi)$ comes from a
cuspidal representation of the quasi-split orthogonal group $\SO(m+1,m-1)$ defined over the quadratic extension associated to $\chi$.
However any transfer theorem for Hecke elements in $(\GL_{2m},\SO(m+1,m-1))$ remains open. 
To get around this situation, we make use of the transfer theorems for some Hecke elements in
 the quadratic base change due to \cite{Yamauchi}. For this, we need the following assumptions on the level $N$:
\begin{enumerate}
\item $N$ is an odd prime or 
\item $N$ is odd and all prime divisors $p_1,\ldots,p_r\ (r\ge 2)$ of $N$ are 
congruent to 1 modulo 4 and 
$\Big(\ds\frac{p_i}{p_j}\Big)=1$ for $i\ne j$,
where $\Big(\ds\frac{\ast}{\ast}\Big)$ denotes the Legendre symbol. 
\end{enumerate}
These conditions are needed in order that for any quadratic extension $M/\Q$ with the conductor $d_M$
dividing $N$, there exists an integral ideal $\frak{N}$ of $M$ such that $\frak{N} \frak{N}^\theta
=(d_M)$ where $\theta$ is the generator of ${\rm Gal}(M/\Q)$. 

\begin{prop}\label{fixedV}
Keep the assumptions on $N$ as above. 
Then 
$$l^{{\rm cusp},{\rm ort}}(2m,N,\tau,\chi)\le 2^{2m\cdot\omega(N)}
{\rm vol}(K^{H}(N))^{-1},
$$
where $H={\rm SO}(m,m)$. 
\end{prop}

\begin{proof} 
Let $M/\Q$ be the quadratic extension associated to $\chi$ and $\mathcal{O}_M$ 
the ring of integers of $M$. Let $\theta$ be the generator of ${\rm Gal}(M/\Q)$.  
Let $K^{{\rm GL}_{2m}}_M(\frak{N})$ be the principal congruence subgroup of $\GL_{2m}(\widehat{\Z}\otimes_\Z
\mathcal{O}_M)$ of the level $\frak{N}$. 
Clearly, the $\theta$-fixed part of $K^{{\rm GL}_{2m}}_M(\frak{N})$ 
is $K^{{\rm GL}_{2m}}(d_M)$ where $d_M$ is the conductor of $M/\Q$ and 
it contains $K^{{\rm GL}_{2m}}(N)$ since $d_M|N$. 
Applying \cite[Theorem 1.6]{Yamauchi}, we have 
for a cuspidal representation $\pi$ of $\GL_{2m}(\A)$ and its base change $\Pi:=
{\rm BC}_{M/\Q}(\pi)$ 
to $\GL_{2m}(\A_M)$,  
$${\rm vol}(K^{{\rm GL}_{2m}}(N))^{-1}{\rm tr}(\pi(1_{K^{{\rm GL}_{2m}}(N)}))=
{\rm vol}(K^{{\rm GL}_{2m}}_M(\frak{N}))^{-1}
{\rm tr}(\Pi(1_{K^{{\rm GL}_{2m}}_M(\frak{N})})).
$$
Recall that our $\pi$ contributing to 
$L^{{\rm cusp},{\rm ort}}(2m,N,\tau,\chi)$ is orthogonal, namely, 
$L(s,\pi, \Sym^2)$ has a pole at $s=1$. 
Note that 
$L(s,\Pi, \Sym^2)=L(s,\pi,\Sym^2)L(s,\pi, \Sym^2\otimes\chi)$.
Now $L(s,\pi \times (\pi\otimes\chi))=L(s,\pi,\wedge^2\otimes\chi)
L(s,\pi,\Sym^2\otimes\chi)$. 
Suppose $\Pi$ is cuspidal. Then $\pi\not\simeq \pi\otimes\chi$. 
So the left hand side has no zero at $s=1$,
and $L(s,\pi,\Sym^2\otimes\chi)$ has no zero at $s=1$.
Therefore, $L(s,\Pi,\Sym^2)$ has a pole at $s=1$.

If $\Pi$ is non-cuspidal, then by \cite{AC}, there exists a cuspidal representation $\tau$ of $\GL_m(\A_M)$ such that 
$$\Pi=\tau\boxplus \tau^\theta.
$$
In such a case, if $m=2$, then $\pi={\rm AI}_M^\Bbb Q \tau$ for some cuspidal 
representation $\tau$ of $\GL_2(\A_M)$; an automorphic induction from $\GL_2(\A_M)$ to $\GL_4(\A_\Q)$. Since $\pi$ is cuspidal and orthogonal, $\tau$ has to be dihedral. 
Such $\pi$ are counted in \cite[Section 2.6]{KWY1} and it amounts to 
$O(N^{\frac{11}{2}+\ve})$ for any $\ve>0$. This will be negligible because ${\rm vol}(K^H(N))\sim c N^{m(2m-1)}=c N^6$ for some constant $c>0$.  
Assume $m\ge 3$. 
It is easy to see that the dimension of 
$\ds\bigoplus_{\Pi:\text{non-cuspidal}}\Pi^{K^{{\rm GL}_{2m}}_M(\frak{N})}_f$ 
is bounded by 
$$O(N^{m^2-1+\frac{m(m+1)}{2}})=O(N^{\frac{3}{2}m^2+\frac{m}{2}-1})
$$
where '$-1$' of $m^2-1$ in the exponent of LHS in the above equation is inserted because of the fixed central character. 
Since ${\rm dim}\, \SO(m,m)=m(2m-1)$ and $m\ge 3$, 
spaces $\Pi^{K^{{\rm GL}_{2m}}_M(\frak{N})}_f$ for which $\Pi$ is non-cuspidal are negligible in the estimation. 
Further, $\Pi$ is orthogonal with trivial central character. [The central character of $\Pi$ is $\chi\circ N_{M/\Bbb Q}=1$.]
Therefore, we can bound $l^{{\rm cusp},{\rm ort}}(2m,N,\tau,\chi)$ by 
$$l^{{\rm cusp},{\rm ort}}(2m,\frak{N},{\rm BC}_{M_\infty/\R}(\tau),1),
$$ 
which is 
similarly defined for cuspidal representations of $\GL_{2m}(\A_M)$. 
Applying the argument of the proof of Proposition \ref{trivial-case} to 
$(\GL_{2m}/M,{\rm SO}(m,m)/M)$, 
the quantity $l^{{\rm cusp},{\rm ort}}(2m,N,\tau,\chi)$ is bounded by 
$2^{2m \omega(\frak{N})}{\rm vol}(K^{H_M}(\frak{N}))^{-1}$ where 
$H_M:={\rm SO}(m,m)/M$ and $\omega(\frak{N})$ denotes the number of 
prime ideals dividing $\frak{N}$. 
The claim follows from $\mathcal{O}_M/\frak{N}\simeq \Z/N\Z$ since 
${\rm vol}(K^{H_M}(\frak{N}))={\rm vol}(K^{H}(N))$ and clearly $\omega(\frak{N})=
\omega(N)$. 
\end{proof}


Note that for any split reductive group $\mathcal{G}$ over $\Q$ and the principal congruence 
subgroup $K^\mathcal{G}(N)$ of level $N$, we see easily that 
${\rm vol}(K^\mathcal{G}(N))\sim c N^{{-\rm dim}\, \mathcal{G}}$ for some constant $c>0$ as $N\to\infty$. Furthermore $\omega(N)\ll \frac {\log N}{\log\log N}$. Hence 
$2^{\omega(N)}\ll N^\epsilon$, and
 $A_n(N)=O(N^\ve)$ and $C_{m_i}(N)=O(N^\ve)$ for each $1\le i\le r$. 

\begin{thm}\label{non-simple} Assume (\ref{suff-reg}). Keep the assumptions on $N$ as in Proposition \ref{fixedV}. 
Then $|HE_{\uk}(N)^{ng}|=O_n(N^{2n^2+n-1+\ve})$ for any $\ve>0$. In particular, 
$$\ds\lim_{N\to \infty}\frac{|HE_{\uk}(N)^{ng}|}{|HE_{\uk}(N)|}=0.
$$ 
\end{thm}
\begin{proof} By Proposition \ref{second-est}, for each partition $\underline{m}=(m_1,\ldots,m_r)$ of $2n+1$, 
we have only to estimate 
$$\frac{A_n(N)}{\varphi(N)} d_{P_{\underline{m}}}(N)\prod_{i=1}^r
l^{{\rm cusp},{\rm ort}}(m_i,N,\tau_{i},\chi_i).
$$
By Proposition \ref{trivial-case} and Proposition \ref{fixedV}, 
$$l^{{\rm cusp},{\rm ort}}(m_i,N,\tau_{i},\chi_i)\ll N^{\frac{m_i(m_i-1)}{2}+\ve},
$$
for any $\ve>0$. 
Further $d_{P_{\underline{m}}(N)}=O(N^{{\rm dim}\, P_{\underline{m}}\backslash \GL_{2n+1}})=O(N^{\sum_{1\le i<j\le r}m_im_j})$. 
Note that $\varphi(N)^{-1}=O(N^{-1+\ve})$ for any $\ve>0$. 
Since 
$$\sum_{1\le i<j\le r}m_i m_j+\sum_{i=1}^r\frac{m_i(m_i-1)}{2}=
\frac{1}{2}\left(\sum_{1\le i,j\le r}m_i m_j \right)-\frac{1}{2}\sum_{i=1}^r m_i=\frac 12(2n+1)^2-\frac 12(2n+1)=2n^2+n,
$$ 
we have the first claim. 

The second claim follows from the dimension formula (\ref{dimension}).
\end{proof}

\section{A notion of newforms in $S_{\uk}(\Gamma(N))$}

In this section, we introduce a notion of a newform in $S_\uk(\Gamma(N))$ with respect to 
principal congruence subgroups.  
Since any local newform theory for $\Sp(2n)$ is unavailable except for $n=1,2$, we 
need a notion of newforms so that we can control a lower bound of conductors for such newforms. This is needed in application to low lying zeros. (See Theorem \ref{stan-conductor} and Lemma \ref{logN}.) 

Recall the description $$S_\uk(\Gamma(N))=\bigoplus_{\psi\in \Psi(G)}
\bigoplus_{\pi\in \Pi_\psi\atop \pi_\infty\simeq \sigma_\uk}m_{\pi,\psi}\pi_f^{K(N)}$$
in terms of Arthur's classification. 

\begin{Def}\label{newspace1} 
The new part (space) of $S_\uk(\Gamma(N))$ is defined by  
$$S^{{\rm new}}_\uk(\Gamma(N))=\bigoplus_{\psi\in \Psi(G)}
\bigoplus_{\pi=\pi_f\otimes \sigma_\uk\in \Pi_\psi\atop 
\pi^{K(N)}\neq 0 \text{ but } \pi^{K(d)}=0 \text{ for any }d\mid N,\ d\neq N}m_{\pi,\psi}\pi_f^{K(N)}.
$$
The orthogonal complement $S^{{\rm old}}_\uk(\Gamma(N))$ of $S^{{\rm new}}_\uk(\Gamma(N))$ in  
$S_\uk(\Gamma(N))$ with respect to Petersson inner product is said to be the old space. 
Let $HE^{{\rm new}}_{\uk}(N)$ be a subset of $HE_\uk(N)$ which is a basis of $S^{{\rm new}}_\uk(\Gamma(N))$.
\end{Def}
\begin{remark}
As the referee pointed out, $S^{{\rm old}}_\uk(\Gamma(N))$ is the intersection of 
$S_\uk(\Gamma(N))$ with the
smallest $G(\A_f )$-invariant space of functions on $G(\Q)\bs G(\A)$ containing 
$S_\uk(\Gamma(M))$
for all proper divisors $M$ of $N$.
\end{remark}

Set $d_p=(1-p^{-1})^n$, $d_M=\prod_{p|M} d_p$, and $C_p=\prod_{j=1}^n(1-p^{-2j})$, $C_M=\prod_{p|M} C_p$. We set $d_1=1$ and $C_1=1$. 

Recall $d_\uk(N)=\dim S_\uk(\G(N))=C_{\uk} \, C_N \, N^{2n^2+n}+O_{\uk}(N^{2n^2})$.

\begin{lem}\label{lem:20220409}
Assume that \eqref{suff-reg} holds and $N$ is square free. 
Then we have 
\[
d_\uk(N)= \sum_{M|N} \dim \, S^{{\rm new}}_\uk(\G(M)) \, \left(\frac{N}{M}\right)^{n^2} \, C_{\frac{N}{M}} \, d_{\frac{N}{M}}^{-1}.
\] 
\end{lem}
\begin{proof}
Let $M\mid N$. 
Take an automorphic representation $\pi=\pi_f\otimes \sigma_\uk$ such that $\dim \pi_f^{K(M)}> 0$ and $\dim \pi_f^{K(L)}= 0$ for any $L\mid M$, $L<M$. 
Under this condition, $\pi$ has an intersection with $S^{{\rm new}}_\uk(\G(M))$, and also with $S_\uk(\G(N))$. 
Let $\pi_f=\otimes_p \pi_p$.
By the assumptions and Theorem \ref{Ram}, for any prime $p \nmid M$, $\pi_p$ is tempered spherical, and so $\pi_p$ is an irreducible induced representation from a Borel subgroup $B$ of $G(\Q_p)$. So $\dim\, \pi_p^{K_p}=1$. Now $K_p/K_p(p)\simeq \Sp_{2n}(\Bbb F_p)$, 
$\#\Sp_{2n}(\Bbb F_p)=p^{2n^2+n}C_p$ and $\# B(\Bbb F_p)=p^{n^2+n}d_p$.
Hence, $\dim \pi_p^{K_p(p)}=p^{n^2} C_p\, d_p^{-1}$ for all $p \nmid M$. 
Since $N$ is square free, this leads to 
\[
\dim  \pi_f^{K(N)}=\dim \pi_f^{K(M)} \times \left(\frac{N}{M}\right)^{n^2} \, C_{\frac{N}{M}} \, d_{\frac{N}{M}}^{-1}. 
\]
Thus, we obtain the assertion. 
\end{proof}

\begin{thm}\label{newforms} Assume that \eqref{suff-reg} holds and $N$ is square free. 
Then we have 
\[
\dim \, S^{{\rm new}}_\uk(\G(N))= C_{\uk} \, C_N \, N^{2n^2+n} \prod_{p|N} \left( 1-d_p^{-1} \, p^{-n^2-n} \right) + O_\uk ( N^{2n^2} ).
\] 
Here $\zeta(n^2)^{-1}< \prod_{p|N} \left( 1-d_p^{-1} \, p^{-n^2-n}\right)< 1$ if $n>1$. If $n=1$, $\prod_{p|N} \left( 1-d_p^{-1} \, p^{-2}\right)>\prod_p \left(1-\frac 1{p(p-1)}\right)=0.374...$
\end{thm}
\begin{proof} Since $C_{\frac NM}=C_N/C_M$ and $d_{\frac NM}=d_N/d_M$,
from Lemma \ref{lem:20220409}, we have
$$d_{\uk}(N) N^{-n^2}C_N^{-1}d_N =\sum_{M|N} \dim S_\uk^\mathrm{new}(\G(M))M^{-n^2}C_M^{-1}d_M.
$$
The M\"obius inversion formula gives
\begin{equation*}
\dim S_\uk^\mathrm{new}(\G(N))N^{-n^2}C_N^{-1}d_N =\sum_{M\mid N} \mu(M)\, d_\uk(N/M)(N/M)^{-n^2}C_{N/M}^{-1}d_{N/M},
\end{equation*}
where $\mu$ denotes the M\"obius function. Therefore,
\begin{equation}\label{eq:20220410}
\dim S_\uk^\mathrm{new}(\G(N)) =\sum_{M\mid N} \mu(M)\, d_\uk(N/M)M^{n^2}C_Md_M^{-1}.
\end{equation}
By \cite[Corollary 1.2]{Wakatsuki}, there exist constants $C_{\uk,r}$ such that 
$d_\uk(N)=\sum_{r=0}^n C_{\uk,r}C_N\, N^{f(r)}$ if $N>2$, where $f(r)=2n^2+n +\frac{r(r-1)}{2}-nr$ and $C_{\uk,0}=C_{\uk}$. 
Further, we take two constants $D_1$ and $D_2$ so that $d_\uk(N)=\sum_{r=0}^n C_{\uk,r} \, C_N \, N^{f(r)} +D_N$ for $N=1$ or $2$. 
Therefore, by \eqref{eq:20220410}, we obtain

\begin{eqnarray*}
&& \dim S_\uk^\mathrm{new}(\G(N)) =\sum_{r=0}^n C_{\uk,r} \, C_N \, N^{f(r)} \sum_{M|N} \mu(M)d_M^{-1}M^{n^2-f(r)}  \\
&&\phantom{xxxxxxx} +\mu(N) \, N^{n^2}\, C_N \, d_N^{-1} \, D_1 + \begin{cases} \mu(N/2) \, (N/2)^{n^2}\, C_{N/2} \, d_{N/2}^{-1} \, D_2 & \text{if $2\mid N$}, \\  0 &  \text{if $2\nmid N$}. \end{cases} 
\end{eqnarray*}
Since $N$ is square free, 
$$\sum_{M|N} \mu(M) d_M^{-1}M^{n^2-f(r)} =\prod_{p|N} \left(1-d_p^{-1} p^{n^2-f(r)}\right).
$$
Therefore,
\begin{eqnarray*}
&& \dim S_\uk^\mathrm{new}(\G(N)) = \sum_{r=0}^n C_{\uk,r} \, C_N \, N^{f(r)} \prod_{p|N} \left( 1-d_p^{-1} p^{-f(r)+n^2} \right)  \\
&&\phantom{xxxxxxx}+\mu(N) \, N^{n^2}\, C_N \, d_N^{-1} \, D_1 + \begin{cases} \mu(N/2) \, (N/2)^{n^2}\, C_{N/2} \, d_{N/2}^{-1} \, D_2 & \text{if $2\mid N$}, \\  0 &  \text{if $2\nmid N$}. \end{cases} 
\end{eqnarray*}
From this we obtain the assertion. 

Now, $d_p< 1$. Hence $\prod_{p|N} \left(1-d_p^{-1} p^{n^2-f(r)}\right)< 1$. Also $d_p^{-1}< p^n$ since $\frac 1{1-p^{-1}}< p$. Therefore, $\prod_{p|N} \left(1-d_p^{-1} p^{-n^2-n}\right)> \prod_{p|N} \left(1-p^{-n^2}\right).$
Here if $n>1$, 
$$\prod_{p|N} \left(1-p^{-n^2}\right)^{-1}< \prod_{p} \left(1-p^{-n^2}\right)^{-1}=\zeta(n^2). 
$$
If $n=1$, $\prod_{p|N} \left(1-d_p^{-1} p^{-n^2-n}\right)=\prod_{p|N} \left(1-\frac 1{p(p-1)}\right)>\prod_p \left(1-\frac 1{p(p-1)}\right)$, which is the Artin constant.
\end{proof}

\section{Equidistribution theorem of Siegel cusp forms; proof of Theorem \ref{main1}}\label{ET}

By the definition in (\ref{mu}), we see that 
$$\widehat{\mu}_{K^S(N),S_1,\xi_{\uk},D_{\underline{l}}^{\rm hol}}(\widehat{h_1})=\frac{{\rm Tr}(T_{h_1}|_{S_{\uk}(\G(N))})}
{{\rm vol}(G(\Q)\bs G(\A))\cdot {\rm dim}\, \xi_\uk}.
$$
Notice that ${\rm dim}\, \xi_\uk=d_{\uk}$ (under a suitable normalization of the measure). 
Applying Theorem \ref{thm:asy} to $S_1$, we have the claim 
by the Plancherel formula of Harish-Chandra: $\widehat{\mu}^{{\rm pl}}_{S_1}(\widehat{h_1})=h_1(1)$.  

\section{Vertical Sato-Tate theorem for Siegel modular forms; Proofs of Theorem \ref{Sato-Tate-tm} and Theorem \ref{finer-ver}}\label{vertical}

Suppose that $\uk=(k_1,\ldots,k_n)$ satisfies the condition (\ref{suff-reg}). Put $\mathbb{T}=\{z\in\C\ |\ |z|=1\}$. 
For $F\in HE_{\underline{k}}(N)$, consider the cuspidal automorphic representation 
$\pi=\pi_F=\pi_\infty\otimes \otimes_p' \pi_{F,p}$ of $G(\A)$ 
associated to $F$. As discussed in the previous section, under the condition (\ref{suff-reg}), 
the $A$-parameter $\psi$ whose $A$-packet contains $\pi$ is semi-simple and 
$\pi_{F,p}$ is tempered for all $p$. Then if $p\nmid N$, $\pi_{F,p}$ is spherical, and 
we can write $\pi_{F,p}$ as $\pi_{F,p}={\rm Ind}^{G(\Q_p)}_{B(\Q_p)}\chi_p$ where $B=TU$ is  the upper Borel subgroup and 
$\chi_p$ is a unitary character on $B(\Q_p)$. For each $1\le j\le n$, put $\alpha_{jp}(\chi_p):=\chi_p(e_j(p^{-1}))$ 
(see (\ref{ejx}) for $e_j(p^{-1})$) and by temperedness we may write 
$\alpha_{jp}(\chi_p)=e^{\sqrt{-1}\theta_j},\ \theta_j\in [0,\pi]$. Let $\widehat{G}=\SO(2n+1)(\C)$ be the complex split orthogonal group over $\C$ 
associated to the anti-diagonal identity matrix. 
Let $\mathcal{L}(\pi_p):W_{\Q_p}\lra \SO(2n+1)(\C)$ 
be the local Langlands parameter given by 
$$\mathcal{L}(\pi_p)({\rm Frob}_p)=(\alpha_{1p}(\chi_p),\ldots,\alpha_{np}(\chi_p),1,\alpha_{1p}(\chi_p)^{-1},
\ldots,\alpha_{np}(\chi_p)^{-1})$$
which is called to be the $p$-Satake parameter.  
Put $a^{(i)}(\chi_p)=a^{(i)}_{F,p}(\chi_p)=\frac 12(\alpha_{ip}(\chi_p)+\alpha_{ip}(\chi_p)^{-1})=\cos\theta_i$ for $1\le i\le n$. 
Let $\widehat{G(\Q_p)}^{{\rm ur,temp}}$ be the isomorphism classes of unramified tempered representations of 
$G(\Q_p)$. 
By \cite[Lemma 3.2]{ST}, we have a topological isomorphism 
$$\widehat{G(\Q_p)}^{{\rm ur,temp}}\stackrel{\sim}{\lra} [0,\pi]^n/\mathfrak S_n=:\Omega
$$
given by 
$$\pi_p={\rm Ind}^{Sp_{2n}(\Q_p)}_{B(\Q_p)}\chi_p\mapsto 
(\arg(a^{(1)}(\chi_p)),\ldots,\arg(a^{(n)}(\chi_p)))=:(\theta_1,\ldots,\theta_n).
$$ 
We denote by $(\theta_1(\pi_{F,p}),\ldots,\theta_n(\pi_{F,p}))\in\Omega$ 
the corresponding element to $\pi_{F,p}$ under the above isomorphism. 
Let $\widehat{B}=\widehat{T}\widehat{U}$ be the upper Borel subgroup of $\widehat{G}=\SO(2n+1)(\Bbb C)$. 
Let $\Delta^+(\widehat{G})$ be the set of all positive roots in $X^\ast(\widehat{T})={\rm Hom}(\widehat{T},\GL_1)$ 
with respect to $\widehat{B}$. 
We view $(\theta_1,\ldots,\theta_n)$ as  
parameters of $\Omega$. Let $\mu^{{\rm pl,temp}}_p$ be the restriction of the Plancherel measure on 
$\widehat{G(\Q_p)}$ to $\widehat{G(\Q_p)}^{{\rm ur,temp}}$ and by abusing the notation we denote by $\mu_p=\mu^{{\rm pl,temp}}_p$ 
its pushforward to $\Omega$.   
Put 
$$t:=(e^{\sqrt{-1}\theta_1},\ldots,e^{\sqrt{-1}\theta_n},1,e^{-\sqrt{-1}\theta_1},
\ldots,e^{-\sqrt{-1}\theta_n})
$$
for simplicity. By \cite[Proposition 3.3]{ST}, we have 
$$\mu_p^{{\rm pl,temp}}(\theta_1,\ldots,\theta_n)=
W(\theta_1,\ldots,\theta_n)d\theta_1\cdots d\theta_n,$$
\begin{eqnarray*}
&&  W(\theta_1,\ldots,\theta_n)=\frac 1{(2\pi)^n} \left(1+\frac 1p\right)^{n^2} \\
&& \cdot\frac{\ds\prod_{\alpha\in \Delta^+(\widehat{G})}|1-e^{\sqrt{-1}\alpha(t)}|^2}
{\ds\prod_{\alpha\in \Delta^+(\widehat{G})}|1-p^{-1}e^{\sqrt{-1}\alpha(t)}|^2}=
\frac{\ds\prod_{i=1}^n|1-e^{\sqrt{-1}\theta_i}|^2
\prod_{1\le i<j\le n\atop \ve=\pm 1}|1-e^{\sqrt{-1}(\theta_i+\ve \theta_j)}|^2}
{\ds\prod_{i=1}^n|1-p^{-1}e^{\sqrt{-1}\theta_i}|^2
\prod_{1\le i<j\le n\atop \ve=\pm 1}|1-p^{-1}e^{\sqrt{-1}(\theta_i+\ve \theta_j)}|^2}.
\end{eqnarray*}
By letting $p\to\infty$, we recover the Sato-Tate measure 
$$\mu^{{\rm ST}}_\infty=\lim_{p\to\infty}\mu^{{\rm pl,temp}}_p = \frac 1{(2\pi)^n}\ds\prod_{i=1}^n |1-e^{\sqrt{-1}\theta_i}|^2
\prod_{1\le i<j\le n\atop \ve=\pm 1} |1-e^{\sqrt{-1}(\theta_i+\ve \theta_j)}|^2\, d\theta_1\cdots d\theta_n. 
$$

Then Theorem \ref{Sato-Tate-tm} and Theorem \ref{finer-ver} follow from Theorem \ref{main1} and Theorem \ref{non-simple}. 

\section{Standard $L$-functions of $\Sp(2n)$}\label{Standard-L}

Let $\uk=(k_1,...,k_n)$ and $F\in HE_\uk(N)$, and $\pi_F$ be a cuspidal representation of $G(\A)$ associated to $F$. 

Assume (\ref{suff-reg}) for $\uk$. 
By (\ref{AD}) and the observation there, the global $A$-packet $\Pi_\psi$ containing $\pi_F$ is associated to 
a semi-simple global $A$ parameter $\psi=\boxplus_{i=1}^r\pi_i$ 
where $\pi_i$ is an irreducible cuspidal representation of $\GL_{m_i}(\A)$. 
Then the isobaric sum $\Pi:=\boxplus_{i=1}^r\pi_i$ 
is an automorphic representation of $\GL_{2n+1}(\A)$. 
Therefore we may define 
$$L(s,\pi_F,{\rm St}):=L(s,\Pi)=\prod_{i=1}^rL(s,\pi_i).$$
Let $L_p(s,\pi_F,{\rm St}):=L(s,\Pi_p)=\prod_{i=1}^rL(s,\pi_{ip})$ be the local $p$-factor 
of $L(s,\pi_F,{\rm St})$ for each rational prime $p$.  

Let $\pi_F=\pi_\infty\otimes\otimes_p' \pi_p$. 
For $p\nmid N$, $\pi_p$ is the spherical representation of $G(\Q_p)$ with the Satake parameter $(\alpha_{1p},\dots,\alpha_{np},1,\alpha_{1p}^{-1},\dots,\alpha_{np}^{-1})$. Then
$$L_p(s,\pi_F,{\rm St})^{-1}=(1-p^{-s})\prod_{i=1}^n (1-\alpha_{ip}p^{-s})(1-\alpha_{ip}^{-1}p^{-s}).
$$
We define the conductor $q(F)$ of $F$ to be the product of the conductors $q(\pi_i)$ of $\pi_i$ 
($1\le i\le r$). 
\begin{thm}
Let $F\in HE_\uk(N)$. Then the standard L-function $L(s,\pi_F,{\rm St})$ has a meromorphic continuation to all of $\Bbb C$. 
 Let
$$\Lambda(s,\pi_F,{\rm St})=q(F)^\frac s2  \, L_\infty(s,\pi_F,{\rm St}) \, L(s,\pi_F,{\rm St}),
$$
where 
$L_\infty(s,\pi_F,{\rm St})=\Gamma_{\Bbb R}(s+\epsilon) \Gamma_{\Bbb C}(s+k_1-1)\cdots \Gamma_{\Bbb C}(s+k_n-n)$, 
$\epsilon=\begin{cases} 0, &\text{if $n$ is even}\\ 1, &\text{if $n$ is odd}\end{cases}$, and $\Gamma_\Bbb R(s)=\pi^{-\frac s2}\Gamma(\frac s2)$, $\Gamma_\Bbb C(s)=2(2\pi)^{-s}\Gamma(s)$.
Then
$$\Lambda(s,\pi_F,{\rm St})=\epsilon(F)\Lambda(1-s,\pi_F,{\rm St}),
$$
where $\epsilon(F)\in\{\pm 1\}$. 
\end{thm}
\begin{proof} It follows from the functional equation of $L(s,\Pi)$ by noting that $\Pi$ is self-dual, and $L(s,\Pi_\infty)=L_\infty(s,\pi_F,{\rm St})$ is the local $L$-function attached to the holomorphic discrete series of the lowest weight $\uk$ (cf. \cite{Kozima}).
\end{proof}

The epsilon factor $\epsilon(F)$ turns out to be always 1. 

\begin{prop} Let $\pi_F$ be associated to a semi-simple $A$-parameter. Then 
$\epsilon(F)=1$.
\end{prop}
\begin{proof} Recall the global $A$-parameter $\psi=\boxplus_{i=1}^r\pi_i$. Let $\omega_i$ be the central 
character of $\pi_i$. 
Since $\pi_i$ is orthogonal, its epsilon factor is $\omega_i(-1)$ by \cite[Theorem 1]{La}. 
Hence $\epsilon(F)=\ds\prod_{i=1}^r\omega_i(-1)=\Big(\prod_{i=1}^{r}\omega_i\Big)(-1)=\trep(-1)=1$ by 
the condition on the central character. 
\end{proof}

\begin{thm}\label{stan-conductor}
For any $F\in HE_\uk(N)$, the conductor $q(F)$ satisfies $q(F)\le N^{2n+1}$.
If $F\in HE^{{\rm new}}_\uk(N)$, then $q(F)\ge \max\left\{N\ds\prod_{p\mid N} p^{-1},\ \prod_{p\mid N}p \right\}.$
So if $F\in HE^{{\rm new}}_\uk(N)$, $q(F)\geq N^{\frac 12}$.
\end{thm}

\begin{proof} 
Let $\pi_F$ be associated to 
a semi-simple global $A$ parameter $\psi=\boxplus_{i=1}^r\pi_i$ 
where $\pi_i$ is an irreducible cuspidal representation of $\GL_{m_i}(\A)$, and let 
$\Pi:=\boxplus_{i=1}^r\pi_i$.
Let $\Pi=\Pi_\infty\otimes\otimes_p' \Pi_p$. By Proposition \ref{transfer}, 
$\Pi$ has a non-zero fixed vector by $K^{GL_{2n+1}}(p^{e_p})$ where 
$e_p={\rm ord}_p(N)$.  
As in the proof of \cite[Lemma 8.1]{KWY}, 
it implies ${\rm depth}(\Pi_p)\leq e_p-1$. Hence $q(\Pi_p)\le p^{(2n+1)e_p}$ by \cite[Proposition 2.2]{LR}. 
Therefore, $q(F)\le N^{2n+1}$.

If $F\in HE^{{\rm new}}_\uk(N)$, by Definition \ref{newspace1},  
it is not fixed by $K^{GL_{2n+1}}(p^{e_p-1})$ for each $p|N$. By \cite[Theorem 1.2]{MY}, 
we have $q(\Pi_p)\ge p^{m_i(e_p-1)}$ for some $i$. 
In particular, $q(\Pi_p)\ge p^{e_p-1}$ for each $p|N$. Hence $q(F)\geq N\prod_{p|N} p^{-1}$.
 It is clear that $q(\Pi_p)\geq p$ if $p|N$. Hence 
$$q(F)\ge
\max\left\{ N\cdot\prod_{p|N} p^{-1},\, \prod_{p|N} p \right\}.
$$
Now, $q(F)^2= q(F)\cdot q(F)\geq N$. Hence our result follows.
\end{proof}

\begin{prop} \label{pole} Keep the assumptions on $N$ as in Proposition \ref{fixedV}. 
Let $F\in HE_\uk(N)$. Then $L(s,\pi_F,{\rm St})$ has a pole at $s=1$ if and only if  $\pi_F$ is associated to 
a semi-simple global $A$-parameter $\psi=1\boxplus \pi_1\boxplus\cdots \boxplus \pi_r$ 
where $\pi_i$ is an orthogonal irreducible cuspidal representation of $\GL_{m_i}(\A)$, such that if $m_i=1$, $\pi_i$ is a non-trivial quadratic character.
Let $HE_\uk(N)^0$ be the subset of $HE_\uk(N)$ such that $L(s,\pi_F,{\rm St})$ has a pole at $s=1$. Then 
$|HE_\uk(N)^0|=O(N^{2n^2-n+\epsilon})$. So $\frac {|HE_\uk(N)^0|}{|HE_\uk(N)|}=O(N^{-2n+\epsilon})$.
\end{prop}
This proves \cite[Hypothesis 11.2]{ST} in our family.
\begin{proof} This follows from the proof of Theorem \ref{non-simple}, by noting that partitions $\underline{m}=(m_1,...,m_r)$ of $2n$ contribute to $HE_\uk(N)^0$.
\end{proof}

In \cite{B1}, B\"ocherer gave the relationship between Hecke operators and $L$-functions for level one and scalar-valued Siegel modular forms and it is extended by Shimura \cite{Shimura} to more general setting.

Let $\underline{a}=(a_1,...,a_n)$, $0\leq a_1\leq \cdots\leq a_n$, and 
$D_{p,\underline{a}}=\diag(p^{a_1},\dots,p^{a_n})$.
Let $F$ be an eigenform in $HE_{\uk}(N)$ with respect to the Hecke operator $T(D_{p,\underline{a}})$ 
for all $p\nmid N$, 
and let $\lambda(F,D_{p,\underline{a}})$ be the eigenvalue. 

Then we have the following identity \cite[Theorem 2.9]{Shimura}: 
\begin{equation}\label{Bo}
\sum_{\underline{a}} \lambda(F,D_{p,\underline{a}}) X^{\sum_{i=1}^n a_i}=\frac {(1-X)}{(1-p^n X)} \prod_{i=1}^n \frac {(1-p^{2i}X^2)}{(1-\alpha_{ip}p^n X)(1-\alpha_{ip}^{-1} p^n X)},
\end{equation}
where $\underline{a}=(a_1,...,a_n)$ runs over $0\leq a_1\leq \cdots\leq a_n$.

Let $\underline{m}=(m_1,...,m_n)$, $m_1|m_2|\cdots | m_n$, and $D_{\underline{m}}=\diag(m_1,...,m_n)$, and let $\lambda(F,D_{\underline{m}})$ be the eigenvalue of the Hecke operator $T(D_{\underline{m}})$. 
Let
$$L^N(s,F)=\sum_{\underline{m}, \, (m_n,N)=1} \lambda(F,D_{\underline{m}}) \det(D_{\underline{m}})^{-s}.
$$
Then
$$L^N(s,F)=\prod_{p\nmid N} L(s,F)_p,\quad L(s,F)_p=\sum_{\underline{a}} \lambda(F,D_{p,\underline{a}}) \det(D_{p,\underline{a}})^{-s}.
$$
It converges for ${\rm Re}(s)> 2n+\frac {k_1+\cdots+k_n}n+1$.

Hence we have 
$$\zeta^N(s)\left[\prod_{i=1}^n \zeta^N(2s-2i)\right] L^N(s,F)=L^N(s-n,\pi_F, {\rm St}),
$$
where $L^N(s,\pi_F,{\rm St})=\prod_{p\nmid N} L_p(s,\pi_F,{\rm St})$, and $\zeta^N(s)=\prod_{p\nmid N} (1-p^{-s})^{-1}$.

The central value of $L^N(s,F)$ is at $s=n+\frac 12$, and $L^N(s,F)$ has a zero at $s=n+\frac 12$ since $L^N(s,\pi_F,{\rm St})$ is holomorphic at $s=\frac 12$. Theorem \ref{thm:asy} implies

\begin{thm}\label{lambda} For $\underline{m}=(m_1,...,m_n)$, $m_1|m_2|\cdots | m_n$ 
with $m_n>1$ and $(m_n,N)=1$,
$$\frac{1}{|HE_{\uk}(N)|}\sum_{F\in HE_{\uk}(N)} \lambda(F,D_{\underline{m}})=O(m_n^{a}N^{-n}),
$$
for some constant $a$. 
\end{thm}
\begin{proof} Let $S_1$ be the set of all prime divisors of $m_n$. 
Since $m_n>1$, $S_1$ is non-empty. 
The main term of RHS in Theorem \ref{thm:asy} includes $h_1(1)$. Clearly, $h_1(1)=0$ because 
the double coset defining the Hecke operator $h_1$ does not contain 
any central elements. 
Since the automorphic counting measure is supported on cuspidal representations, 
Theorem \ref{thm:asy} implies the claim.
\end{proof}

Write $L^N(s,F)=\ds\sum_{m=1\atop (m,N)=1}^\infty a_F(m)m^{-s}$, and $L(s,F)_p=\ds\sum_{k=0}^\infty a_F(p^k) p^{-ks}$ for each prime $p\nmid N$. Here $a_F(p^k)=\sum_{\underline{a}} \lambda(F,D_{p,\underline{a}})$, where the sum is over all $\underline{a}=(a_1,...,a_n)$ such that $0\leq a_1\leq\cdots\leq a_n$, $a_1+\cdots+a_n=k$. Hence, for $k>0$, $p\nmid N$,
$$\frac 1{d_\uk(N)} \sum_{F\in HE_{\uk}(N)} a_F(p^k)=O(p^{ka}N^{-n}).
$$
More generally,

\begin{cor}\label{a_F} For $m>1$, $(m,N)=1$,
$$\frac 1{d_\uk(N)} \sum_{F\in HE_{\uk}(N)} a_F(m)=O(m^{a}N^{-n}).
$$
\end{cor}
\begin{proof} We have $a_F(m)=\sum_{\underline{m}} \lambda(F,D_{\underline{m}})$, where the sum is over all $\underline{m}=(m_1,...,m_n)$, $m_1|m_2|\cdots |m_n$, $m_1\cdots m_n=m$. Our assertion follows from Theorem \ref{lambda}.
\end{proof}

Write 
$L^N(s,\pi_F,{\rm St})=\ds\sum_{m=1\atop (m,N)=1}^\infty \mu_F(m)m^{-s}$. 
Then from (\ref{Bo}), we have, for $p\nmid N$,
$$\mu_F(p)=(a_F(p)+1)p^{-n},\quad \mu_F(p^2)=1+p^{-2}+\cdots+p^{-2n}+(a_F(p^2)+a_F(p))p^{-2n}.
$$
More generally, for $p\nmid N$,

$$\mu_F(p^k)=\begin{cases} 1+p^{-2} h_k(p^{-2})+p^{-n}\ds\sum_{i=1}^k h_{ik}(p^{-1})a_F(p^i), &\text{if $k$ is even}\\
p^{-n} h_k'(p^{-2})+p^{-n}\ds\sum_{i=1}^k h_{ik}'(p^{-1}) a_F(p^i), &\text{if $k$ is odd},
\end{cases}
$$
where $h_k,h_k', h_{ik}, h_{ik}'\in\Bbb Z[x]$. Therefore, for $(m,N)=1$,
$$
\mu_F(m)=\prod_{p|m} (\delta_{p,m}+p^{-2}h_m^\delta(p^{-1}))+ \sum_{u|m\atop u>1} A_u a_F(u),
$$
where $A_u\in\Bbb Q$, $h_m^\delta\in\Bbb Z[x]$, and $\delta=\delta_{p,m}=\begin{cases} 1, &\text{if $v_p(m)$ is even}\\0, &\text{otherwise}\end{cases}$.

Therefore, by Corollary \ref{a_F}, we have

\begin{thm}\label{stan} Fix $\uk=(k_1,...,k_n)$ and let $m=\ds\prod_{p|m} p^{v_p(m)}$ which is 
coprime to $N$. Then 
$$\frac{1}{d_{\uk}(N)}\sum_{F\in HE_{\uk}(N)} \mu_F(m) = \prod_{p|m} \left(\delta_{p,m} + p^{-2}h_m^\delta(p^{-1})\right)+ O(N^{-n}m^c).
$$
\end{thm}

This proves \cite[Conjecture 6.1 in level aspect]{KWY1} for the $\Sp(4)$ case. 

\section{$\ell$-level density of standard $L$-functions}\label{r-level}

In this section, we assume (\ref{suff-reg}) and keep the assumptions on $N$ in Proposition \ref{fixedV}. 
Then we show unconditionally that the $\ell$-level density ($\ell$ a positive integer) of the standard $L$-functions of the family 
$HE_\uk(N)$ has the symmetry type $Sp$ in the level aspect. Shin and Templier \cite{ST} showed it under several hypotheses with a family which includes non-holomorphic forms.

Under assumption (\ref{suff-reg}), $F$ satisfies the Ramanujan conjecture, namely, $|\alpha_{ip}|=1$ for each $i$.
Let
$$-\frac {L'}L(s,\pi_F,{\rm St})=\sum_{m=1}^\infty \Lambda(m)b_F(m) m^{-s},
$$
where $b_F(p^m)=1+\alpha_{1p}^m+\cdots +\alpha_{np}^m+\alpha_{1p}^{-m}+\cdots+\alpha_{np}^{-m}$ when $\pi_p$ is spherical.

For $F\in HE_{\uk}(N)$, let $\Pi$ be the Langlands transfer of $\pi_F$ to ${\rm GL}_{2n+1}$. If $F\in HE_{\uk}(N)^g$, then $L(s,\Pi,\wedge^2)$ has no pole at $s=1$, and $L(s,\Pi,Sym^2)$ has a simple pole at $s=1$. Let $L(s,\Pi\times\Pi)=\sum \lambda_{\Pi\times\Pi}(n)n^{-s}, L(s,\Pi,\wedge^2)=\sum \lambda_{\wedge^2(\Pi)}(n)n^{-s}, L(s,\Pi,Sym^2)=\sum \lambda_{Sym^2(\Pi)}(n)n^{-s}$. Then $\mu_F(p^2)=\lambda_{Sym^2(\Pi)}(p)$, and 
$\mu_F(p)^2=\lambda_{\Pi\times\Pi}(p)=\lambda_{\wedge^2(\Pi)}(p)+\lambda_{Sym^2(\Pi)}(p)$.

Note that $\mu_F(p)=b_F(p)$, and $b_F(p^2)=2\mu_F(p^2)-\mu_F(p)^2$. 
Let $T(p,\underline{a})=\Gamma(N) \begin{pmatrix} D_{p,\underline{a}}&0\\0&D_{p,\underline{a}}^{-1}\end{pmatrix}\Gamma(N)$. 
By Theorem \ref{main-appendix}, 
 $T(p,(\overbrace{0,...,0}^{n-1},1))^2$ is a linear combination of
$$T(p,(\overbrace{0,...,0}^{n-1},2)),\ 
 T(p,(\overbrace{0,...,0}^{n-2},1,1)),\ T(p,(\overbrace{0,...,0}^{n-1},1)),\ 
 T(p,\overbrace{(0,...,0)}^{n})=\Gamma(N)I_{2n}\Gamma(N).
$$ 
Therefore, 
 by Theorem \ref{stan}, if $p\nmid N$,
$\frac 1{d_\uk(N)} \sum_{F\in HE_{\uk}(N)} \mu_F(p)^2$ is of the form $1+p^{-1}g(p^{-1})+O(p^c N^{-n})$ 
for some polynomial $g\in\Bbb Z[x]$ and $c>0$. 
Here the main term $1+p^{-1}g(p^{-1})$ comes from the coefficient $p\sum_{i=0}^{2n-1} p^i$ of 
$T(p,\overbrace{(0,...,0)}^{n})$ in the linear combination. Here the explicit determination of the coefficient is necessary in our application. Hence we have 

\begin{prop} \label{b_F} For some $a>0$, and $p\nmid N$, 
\begin{eqnarray*}
&& \frac{1}{d_{\uk}(N)}\sum_{F\in HE_{\uk}(N)} b_F(p) = O(p^{-1})+O(p^a N^{-n}),\\
&&\frac{1}{d_{\uk}(N)}\sum_{F\in HE_{\uk}(N)} b_F(p^2) = 1+O(p^{-1})+O(p^a N^{-n}).
\end{eqnarray*}
\end{prop}

We denote the non-trivial zeros of $L(s,\pi_F,{\rm St})$ by $\sigma_{F,j}=\frac{1}{2}+\sqrt{-1} \gamma_{F,j}$. Without assuming the GRH for $L(s,\pi_F,{\rm St})$, we can order them as
\begin{equation*}
\cdots \leq Re({\gamma_{F,-2}}) \leq Re({\gamma_{F,-1}}) \leq 0 \leq Re({\gamma_{F,1}}) \leq Re({\gamma_{F,2}}) \leq \cdots
\end{equation*}
Let $c(F)=q(F)(k_1\cdots k_n)^2$ be the analytic conductor, and
$\displaystyle \log c_{\uk,N}=\frac 1{d_{\uk}(N)} \sum_{F\in HE_{\uk}(N)} \log c(F)$.
From Theorems \ref{newforms} and \ref{stan-conductor}, we have

\begin{lem}\label{logN} Let $n>1$. We assume that $N$ is square free. Then
$(k_1\cdots k_n)^2 N^{\frac 1{2\zeta(n^2)}}\leq c_{k,N}\leq (k_1\cdots k_n)^2 N^{2n+1}$.
\end{lem}

This proves \cite[Hypothesis 11.4]{ST} in our family. It is used in the proof of (\ref{one-level}).

\begin{proof} By Theorem \ref{stan-conductor}, $q(F)\leq N^{2n+1}$. It gives rise to the upper bound.
If $F\in HE^{{\rm new}}_\uk(N)$, $q(F)\geq N^{\frac 12}$ by Theorem \ref{stan-conductor}. By Theorem \ref{newforms}, $|HE^{{\rm new}}_\uk(N)|\geq \zeta(n^2)^{-1} |HE_\uk(N)|$.
Hence 
$$\log c_{\uk,N}\geq \log(k_1\cdots k_n)^2+\frac 1{d_{\uk}(N)} \sum_{F\in HE^{\rm new}_{\uk}(N)} \log q(F)\geq \log(k_1\cdots k_n)^2+\frac 1{2\zeta(n^2)} \log N.
$$
\end{proof}

Consider, for an even Paley-Wiener function $\phi$,
$$D(F,\phi)=\sum_{\gamma_{F,j}} \phi\left(\frac {\gamma_{F,j}}{2\pi}\log c_{\uk,N}\right).
$$
Then as in \cite[(9.1)]{KWY},
\begin{eqnarray*}
&& \frac 1{d_{\uk}(N)} \sum_{F\in HE_\uk(N)} D(F,\phi)=\widehat{\phi}(0)-\frac 12 \phi(0) \\
&& \phantom{xxxxxxxxxx} - \frac{2}{(\log c_{\uk,N}) d_{\uk}(N)} \sum_{F\in HE_\uk(N)} \sum_p \frac{b_F(p)\log p}{\sqrt{p}}\widehat{\phi}\left( \frac{\log p}{\log c_{\uk,N}}\right) \nonumber \\
&& \phantom{xxxxxxxxxx} -\frac{2}{(\log c_{\uk,N}) d_{\uk}(N)} \sum_{F\in HE_\uk(N)} \sum_p \frac{(b_F(p^2)-1)\log p}{p}\widehat{\phi}\left( \frac{2\log p}{\log c_{\uk,N}}\right) \\
&& \phantom{xxxxxxxxxx} +O\left(\frac {|HE_\uk(N)^0|}{d_{\uk}(N)}\right)+O\left( \frac{1}{\log c_{\uk,N}} \right),
\end{eqnarray*}
where $HE_\uk(N)^0$ is in Proposition \ref{pole}. [In \cite[(9.4)]{KWY}, the term $O\left(\frac {|HE_\uk(N)^0|}{d_{\uk}(N)}\right)$ was omitted.]

By Proposition \ref{b_F}, we can show as in \cite{KWY} that for an even Paley-Wiener function $\phi$ such that the Fourier transform $\hat{\phi}$ of $\phi$ is supported in $(-\beta,\beta)$, where $\ds\beta=\min\left(\frac n{(2n+1) (a+1/2)}, \frac {2n}{(2n+1) a}\right)$, 

\begin{equation}\label{one-level}
\frac 1{d_{\uk}(N)} \sum_{F\in HE_{\uk}(N)}  D(F,\phi)=\widehat{\phi}(0)-\frac 12 \phi(0)
+O\left( \frac{1}{\log c_{\uk,N}} \right)
=\int_\Bbb R \phi(x)W(\text{Sp})(x)\, dx+O\left( \frac{\omega(N)}{\log N} \right),
\end{equation}
where $\omega(N)$ is the number of prime factors of $N$, and $W(\text{Sp})(x) = 1- \dfrac {\sin 2\pi x}{2\pi x}$. [When we exchange two sums, if $p\nmid N$, we use Proposition \ref{b_F}. If $p|N$, by the Ramanujan bound, $|b_F(p)|\leq n, |b_F(p^2)|\leq n$. Hence by the trivial bound, we would obtain $\sum_{p|N}  \frac{b_F(p)\log p}{\sqrt{p}}\ll \omega(N)$ and $\sum_{p|N}  \frac{b_F(p^2)\log p}p\ll \omega(N)$.]

For a general $\ell$, let 
$$W(\text{Sp})(x) = \text{det}(K_{-1}(x_j,x_k))_{1\leq j\leq \ell\atop 1\leq k\leq \ell},
$$
where $K_{-1}(x,y)=\dfrac {\sin \pi(x-y)}{\pi(x-y)}- \dfrac {\sin \pi(x+y)}{\pi(x+y)}$.
Let $\phi(x_1,...,x_\ell)=\phi_1(x_1)\cdots \phi_\ell(x_\ell)$, where each $\phi_i$ is an even Paley-Wiener function and $\hat \phi(u_1,...,u_\ell)=\hat \phi_1(u_1)\cdots \hat\phi_\ell(u_\ell)$.
We assume that the Fourier transform $\hat{\phi_i}$ of $\phi_i$ is supported in $(-\beta,\beta)$ for $i=1,\dots,\ell$.
The $\ell$-level density function is
\begin{equation*}\label{n-level-st}
D^{(\ell)}(F, \phi)
={\sum}_{j_1,\cdots,j_\ell}^*\phi\left(\gamma_{j_1}\frac{\log c_{\uk,N}}{2 \pi},\gamma_{j_2}\frac{\log c_{\uk,N}}{2 \pi},\dots,\gamma_{j_\ell}\frac{\log c_{\uk,N}}{2 \pi}\right)
\end{equation*}
where $\sum_{j_1,...,j_\ell}^*$ is over $j_i=\pm 1,\pm 2,...$ with $j_{a}\ne \pm j_{b}$ for $a\ne b$.
Then as in \cite{KWY1}, using Theorem \ref{stan}, we can show 

\begin{thm}\label{n-level-sp} We assume that $N$ is square free. 
Let $\phi(x_1,...,x_\ell)=\phi_1(x_1)\cdots \phi_\ell(x_\ell)$, where each $\phi_i$ is an even Paley-Wiener function and $\hat \phi(u_1,...,u_\ell)=\hat \phi_1(u_1)\cdots \hat\phi_\ell(u_\ell)$. Assume the Fourier transform $\hat{\phi_i}$ of $\phi_i$ is supported in $(-\beta,\beta)$ for $i=1,\cdots,\ell$. (See (\ref{one-level}) for the value of $\beta$.)
Then
$$
\frac 1{d_{\uk}(N)} \sum_{F \in HE_{\uk}(N)} D^{(\ell)}(F,\phi)=
\int_{\Bbb R^\ell} \phi(x)W({\rm Sp})(x)\, dx + O\left(\frac {\omega(N)}{\log N}\right).
$$
\end{thm}

\section{The order of vanishing of standard $L$-functions at $s=\frac 12$}\label{order}

In this section, we show that the average order of vanishing of standard $L$-functions at $s=\frac 12$ is bounded under GRH (cf. \cite{ILS, Brumer}). Under GRH on $L(s,\pi_F,{\rm St})$, its zeros are $\frac 12+\sqrt{-1}\gamma_F$ with $\gamma_F\in\Bbb R$.

\begin{thm}\label{order} Assume the GRH. Assume (\ref{suff-reg}) and $N$ is square free.  Let $r_F={\rm ord}_{s=\frac 12} L(s,\pi_F,{\rm St})$.
Then
$$\frac 1{d_\uk(N)} \sum_{F\in HE_{\uk}(N)} r_F\leq  C,
$$ 
for some constant $C>0$.
\end{thm}
\begin{proof}
Choose $\displaystyle\phi(x)=\left(\frac {2\sin \frac {x \beta}2}x\right)^2$ for $x\in\Bbb R$, where 
$\beta$ is from \eqref{one-level}. Then 
$$\widehat\phi(x)=\begin{cases} \beta-|x|, &\text{if $|x|<\beta$}\\ 0, &\text{otherwise}\end{cases}.
$$ 
Since $\phi(x)\geq 0$ for $x\in\Bbb R$, from (\ref{one-level}),
we have
$$\frac 1{d_\uk(N)} \sum_{F\in HE_\uk(N)} r_F\phi(0)\leq \widehat\phi(0)-\frac 12\phi(0)+O\left(\frac 1{\log\log N}\right).
$$
Hence we have 
$$
\frac 1{d_\uk(N)} \sum_{F\in HE_\uk(N)} r_F\leq \frac 1{\beta}-\frac 12+O\left(\frac 1{\log\log N}\right).
$$
\end{proof}

We can show a similar result for the spinor $L$-function of $\GSp(4)$. Recall the following from \cite{KWY}:

\begin{prop} Assume $(N,11!)=1$. 
\begin{enumerate} 
  \item (level aspect) Fix $k_1,k_2$. 
  Then for $\phi$ whose Fourier transform $\hat\phi$ has support in $(-u,u)$ for some $0<u<1$, as $N\to\infty$ (See \cite[Proposition 9.1]{KWY} for the value of $u$),
$$
 \frac 1{d_\uk(N)} \sum_{F\in HE_{\underline{k}}(N)}  D(\pi_F,\phi, {\rm Spin})=\hat\phi(0)+\frac 12 \phi(0)+O(\frac 1{\log\log N}).
$$
  \item (weight aspect) Fix $N$. Then for $\phi$ whose Fourier transform $\hat\phi$ has support in $(-u,u)$ for some $0<u<1$, as $k_1+k_2\to\infty$,
$$\frac 1{d_\uk(N)} \sum_{F\in HE_{\underline{k}}(N)}  D(\pi_F,\phi, {\rm Spin})=\hat\phi(0)+\frac 12 \phi(0)+
O(\frac 1{\log ((k_1-k_2+2)k_1k_2)}).
$$
\end{enumerate}
\end{prop}

As in Theorem \ref{order}, we have

\begin{thm} Let $G=\GSp(4)$. Assume the GRH, and let $r_F={\rm ord}_{s=\frac 12} L(s,\pi_F,{\rm Spin})$. Then
$$\frac 1{d_\uk(N)} \sum_{F\in HE_{\underline{k}}(N)} r_F\leq \begin{cases} \frac 1u+\frac 12+O(\frac 1{\log\log N})
, &\text{level aspect}\\ \frac 1u+\frac 12 +
O(\frac 1{\log ((k_1-k_2+2)k_1k_2)})., &\text{weight aspect}\end{cases}.
$$ 
\end{thm}

\bigskip

\section{Appendix} 
In this appendix we compute the product $T(p,(\overbrace{0,...,0}^{n-1},1))^2$ in Section \ref{r-level}. 

\begin{thm}\label{main-appendix}
\begin{eqnarray*}
&& T(p,(\overbrace{0,...,0}^{n-1},1))^2=T(p,(\overbrace{0,...,0}^{n-1},2))+(p+1) T(p,(\overbrace{0,...,0}^{n-2},1,1))+(p^n-1)T(p,(\overbrace{0,...,0}^{n-1},1)) \\
&&\phantom{xxxxxxxxxxxxx} +\left(p\sum_{i=0}^{2n-1}p^i\right) T(p,\overbrace{(0,...,0)}^{n}).
\end{eqnarray*}
\end{thm}

This agrees with \cite[(2.7), p.362]{KWY} when $n=2$. [Note that
the coefficient of $R_{p^2}$ there should be replaced with $p^4+p^3+p^2+p$.]

Since $p\nmid N$, we work on $K=\Sp(2n,\Z_p)$ instead of $\G(N)$. Put 
$$T_{p,n-1}:=pT(p,(0,\ldots,0,1))=K\diag(1,\overbrace{p,\ldots,p}^{n-1},p^2,
\overbrace{p,\ldots,p}^{n-1})K\in \GSp(2n,\Bbb Q_p).
$$
It suffices to consider $T_{p,n-1}^2$. 
Let us first compute the coset decomposition. 
Put $\Lambda=\GL_n(\Z_p)$ where the identity element is denoted by $1_n$. 
For any ring $R$, let $S_n(R)$ be the set of 
all symmetric matrices of size $n$ defined over $R$ and 
$M_{m\times n}(R)$ be the set of matrices of size $m\times n$ defined over $R$. 
Put $M_n(R)=M_{n\times n}(R)$ for simplicity. For each $D\in M_n(\Z_p)$ we define 
$$B(D):=\{B\in M_n(\Z_p)\ |\ {}^tB D={}^tDB\}.$$
For each $B_1,B_2\in B(D)$, we write $B_1\sim B_2$ if there exists $M\in M_n(\Z_p)$ 
such that $B_1-B_2=MD$.  
We denote by $B(D)/\sim$ the set of all equivalence classes of $B(D)$ by the relation $\sim$. 
We regard $\mathbb{F}_p$ (resp. $\Z/p^2\Z$) as the subset $\{0,1,\ldots,p-1\}$ 
(resp. $\{0,1,\ldots,p^2-1\}$) of $\Z$. 
Let $D_I$ be the set of the following matrices in $M_n(\Z_p)$:
$$D^I_{n-1}={\rm diag}(
\overbrace{p,\ldots,p}^{n-1},1),\ D^I_s=D^I_{s}(x):=
\left(\begin{array}{c|c|c}
p\cdot 1_s &  &  \\
\hline
& 1 & x \\
\hline
 & &p\cdot 1_{n-1-s}
\end{array}
\right),\ 0\le s\le n-2,\ x\in M_{1\times(n-1-s)}(\mathbb{F}_p)
$$
where we fill out zeros in the blank blocks. 
The cardinality of $D_I$ is $1+p+\cdots+p^{n-1}=\ds\frac{p^n-1}{p-1}$ which is 
equal to that of $\Lambda\bs \Lambda d_{n-1}\Lambda$ where 
$d_{n-1}=\diag(1,\overbrace{p,\ldots,p}^{n-1})$. 
Similarly, let $D_{II}$ be the set of the following matrices:
$$D^{II}_{n-1}={\rm diag}(p,
\overbrace{1,\ldots,1}^{n-1}),\ D^{II}_s=D^{II}_{s}(y):=
\left(\begin{array}{c|c|c}
1_s & y &  \\
\hline
& p &  \\
\hline
 & &1_{n-1-s}
\end{array}
\right),\ 1\le s\le n-1,\ y\in M_{s\times 1}(\mathbb{F}_p)
$$
The cardinality of $D_{II}$ is $1+p+\cdots+p^{n-1}=\ds\frac{p^n-1}{p-1}$ which is 
equal to that of $\Lambda\bs \Lambda d_{1}\Lambda$ where 
$d_{1}=\diag(\overbrace{1,\ldots,1}^{n-1},p)$. 
Finally for each $M\in M_n(\Z_p)$ we denote by $r_p(M)$ the rank of $M$ mod $p\Z_p$. 
\begin{lem}\label{DC}Assume $p$ is odd. 
The right coset decomposition $T_{p,n-1}=\ds\coprod_{\alpha\in J}K \alpha$ 
consists of the following elements:
\begin{enumerate}
\item (type I) $\alpha=\alpha_I(D,B)=
\begin{pmatrix} 
p^2\cdot {}^tD^{-1} & B \\
0_n & D
\end{pmatrix}
$ where $D$ runs over the set $D_I$ and $B$ runs over complete representatives of  
$B(D)/\sim$ such that $r_p(\alpha)=1$. Further, for each $D^I_s$, $B$ can be taken 
over 
\begin{itemize}
\item if $s\neq 0$, then $x\neq 0$ and $B=0$;
\item if $s=0$, then $x=0$ and $B=0$. 
\end{itemize} 
\item (type II) 
 $\alpha=\alpha_{II}(D,B)=
\begin{pmatrix} 
p\cdot {}^tD^{-1} & B \\
0_n & pD
\end{pmatrix}
$ 
where $D$ runs over the set $D_{II}$ and $B$ runs over complete representatives of  
$B(D)/\sim$ such that $r_p(\alpha)=1$. Further, for each $D^{II}_s$, $B$ can be taken 
over 
\begin{itemize}
\item if $s=0$, then $\begin{pmatrix} 
B_{22} & B_{23}\\
p\cdot {}^t B_{23} & 0_{n-1}
\end{pmatrix}$ where $B_{22}$ runs over $\Z/p^2\Z$ and 
$B_{23}$ runs over $M_{1\times (n-1)}(\F_p)$;
\item if $s\neq 0$, for $D^{II}_s(y),\ y\in M_{s\times 1}(\F_p)$,  
$\left(\begin{array}{c|c|c}
0_s & p\cdot {}^tB_{21}  & 0_{s\times (n-1-s)}  \\
\hline
B_{21}& B_{22} & B_{23}  \\
\hline
0_{(n-1-s)\times s} & p\cdot {}^t B_{23} &0_{n-1-s}
\end{array}
\right)$ where $B_{21},B_{22}$ and $B_{23}$ run over $M_{1\times s}(\F_p),\ 
\Z/p^2\Z$, and $M_{1\times t}(\F_p)$ respectively.
\end{itemize} 
\item (type III) $\alpha=\alpha_{III}(B)=
\begin{pmatrix} 
p1_n & B \\
0_n & p1_n
\end{pmatrix}
$ where $B$ runs over $S_n(\mathbb{F}_p)$ with $r_p(B)=1$. 
The number  of such $B$'s is $p^n-1$.  
\end{enumerate}
\end{lem}
\begin{proof}
We just apply the formula \cite[(3.94), p. 98]{Andrianov}. 
First we need to compute a complete system of representatives of 
$\Lambda\bs \Lambda t\Lambda \simeq (t^{-1}\Lambda t)\cap \Lambda\bs \Lambda$ 
for each $t\in \{d_{n-1},d_1,p1_{n}\}$ where $d_{n-1}=\diag(1,\overbrace{p,\ldots,p}^{n-1})$ and 
$d_{1}=\diag(\overbrace{1,\ldots,1}^{n-1},p)$. By direct computation, for $t=d_{n-1}$ (resp. 
$t=d_1$), it is given by $D^I$ 
(resp. $D^{II}$). For $t=p\cdot 1_{n}$ it is obviously a singleton. 

As for the computation of $B(D)/\sim$, we give details only for $D\in D^I$ and 
the case of $D^{II}$ is similarly handled. 
For each $D=D^I_s(x),\ 0\le s\le n-2 $, put 
$A_s=\left(\begin{array}{c|c|c}
1_s &  &  \\
\hline
& 1 & -px  \\
\hline
 & &1_{n-1-s}
\end{array}
\right)$ so that $DA_s=\left(\begin{array}{c|c|c}
p\cdot 1_s &  &  \\
\hline
& 1 &  \\
\hline
 & &p\cdot 1_{n-1-s}
\end{array}
\right)$. Put $A_{n-1}=1_{2n}$ for $D=D^I_{n-1}$. 
Then for each $D=D^I_s$ we have a bijection 
$$B(D)/\sim\stackrel{\sim} {\longrightarrow} B(DA_s)/\sim,\quad B\mapsto BA_s.
$$
Therefore, we may compute $B(DA_s)/\sim$ and convert them by multiplying 
$A^{-1}_s$ on the right. 

We write $B\in B(DA_s)$ as a block matrix 
$B=\left(\begin{array}{c|c|c}
\overbrace{B_{11}}^{s} & \overbrace{B_{12}}^{1}   & \overbrace{B_{13}}^{n-1-s}   \\
\hline
B_{21}& B_{22} & B_{23}  \\
\hline
B_{31} & B_{32} &B_{33}
\end{array}
\right)$ 
with respect to the partition $s+1+(n-1-s)$ of $n$ where the column is also decomposed 
as in the row. The relation yields 
$$B=\left(\begin{array}{c|c|c}
B_{12 }&B_{12}   & B_{13}   \\
\hline
p\cdot{}^t B_{12}& B_{22} & p\cdot {}^t B_{32}  \\
\hline
{}^t B_{13} & B_{32} &B_{33}
\end{array}
\right)$$
where $B_{11}\in S_{s}(\Z_p)$,\ $B_{22}\in \Z_p$, and $B_{33}\in S_{n-1-s}(\Z_p)$. 
We write $X\in M_n(\Z_p)$ as 
$\left(\begin{array}{c|c|c}
\overbrace{X_{11}}^{s} & \overbrace{X_{12}}^{1}   & \overbrace{X_{13}}^{n-1-s}   \\
\hline
X_{21}& X_{22} & X_{23}  \\
\hline
X_{31} & X_{32} &X_{33}
\end{array}
\right)$ 
with respect to the partition $s+1+(n-1-s)$ of $n$ as we have done for $B$. 
Then 
$$XDA_s=
\left(\begin{array}{c|c|c}
pX_{11} &X_{12} & pX_{13}  \\
\hline
pX_{21}& X_{22} & pX_{23}  \\
\hline
pX_{31} & X_{32} &pX_{33}
\end{array}
\right)$$
Our matrix $B$ in $B(DA_s)/\sim$ is considered by taking modulo $XDA_s$ 
for any $X\in M_n(\Z_p)$. Hence $B$ can be, up to equivalence, of form 
\begin{equation}\label{B(D)}B=\left(\begin{array}{c|c|c}
B_{11} &0_{s\times 1} & B_{13} \\
\hline
0_{1\times s}& 0 & 0_{1\times (n-1-s)}  \\
\hline
{}^t B_{13} & 0_{(n-1-s)\times 1} & B_{33}
\end{array}
\right)
\end{equation}
where $B_{11},B_{33}$, and $B_{13}$ belong to $S_s(\F_p)$, 
$S_{n-1-s}(\F_p)$, and $M_{s\times(n-1-s)}(\F_p)$ respectively. 
Further, to multiply $A^{-1}_s$ on the right never change anything. 
Therefore, (\ref{B(D)}) gives a complete system of representatives of $B(D)/\sim$ for 
$D=D^I_s$. The condition $r_p(\alpha_I(D,B))=1$ and the modulo $K$ on the left yield the desired result. 
For each $D\in D^{II}_s$, a similar computation shows 
any element of $S(p\cdot D)/\sim$ is given by
$$\left(\begin{array}{c|c|c}
\overbrace{B_{11}}^{s} & \overbrace{p\cdot {}^t B_{21}}^{1}   & \overbrace{B_{13}}^{n-1-s}   \\
\hline
B_{21}& B_{22} & B_{23}  \\
\hline
{}^t B_{13} & p\cdot {}^t B_{23} &B_{33}
\end{array}
\right)$$
modulo under the matrices of forms 
$$\left(\begin{array}{c|c|c}
pX_{11} &p^2X_{12} & pX_{13}  \\
\hline
pX_{21}& p^2X_{22} & pX_{23}  \\
\hline
pX_{31} & p^2 X_{32} &pX_{33}
\end{array}
\right).$$
Therefore, $B_{11},B_{13},B_{21},B_{22},B_{23}$, and $B_{33}$ run over 
$$M_{s}(\F_p),\ M_{s\times (n-1-s)(\F_p)},\ M_{1\times s(\F_p)},\ \Z/p^2\Z,\ 
M_{1\times (n-1-s)(\F_p)},$$ and $M_{n-1-s}(\F_p)$ respectively.   
The claim now follows from the rank condition $r_p(\alpha_{II}(D,B))=1$ 
 and the modulo $K$ on the left again. 

As for $D=p1_{n}$ in the case of type III, it is easy to see that 
$S(D)/\sim$ is naturally identified with $S_n(\F_p)$.  
Recall $p$ is an odd prime by assumption. 
The number of matrices in $S_n(\F_p)$ of rank 1 is given in \cite[Theorem 2]{Mac}. 
\end{proof}

Recall the right coset decomposition $T_{p,n-1}:=K\diag(1,\overbrace{p,\ldots,p}^{n-1},p^2,
\overbrace{p,\ldots,p}^{n-1})K=\coprod_{\alpha\in J}K\alpha$. 
For each $\alpha,\beta\in J$, we observe that any element of 
$K\alpha \beta K$ is of mod $p$ rank at most two and has the similitude $p^4$. 
Hence the double coset $K\alpha \beta K$ satisfies $K\alpha \beta K=
K\gamma K$, where $\gamma$ is one of the following 4 elements:
$$
\gamma_1:=\diag(1,\overbrace{p^2,\ldots,p^2}^{n-1},p^4,
\overbrace{p^2,\ldots,p^2}^{n-1}),\ \gamma_2:=\diag(p,p,\overbrace{p^2,\ldots,p^2}^{n-2},p^3,p^3,
\overbrace{p^2,\ldots,p^2}^{n-2}),$$
$$\gamma_3:=\diag(p,\overbrace{p^2,\ldots,p^2}^{n-1},p^3,
\overbrace{p^2,\ldots,p^2}^{n-1}),\ \gamma_4:=p^2\cdot I_{2n}
$$
Here we use the Weyl elements in $K$ to renormalize the order of entries. Then
\begin{equation}\label{product}
T_{p,n-1}\cdot T_{p,n-1}=\sum_{i=1}^4 m(\gamma_i) K\gamma_i K
\end{equation}
where $m(\gamma_i)$ is defined by 
\begin{equation}\label{m-gamma}
m(\gamma_i):=|\{(\alpha,\beta)\in J\times J\ |\ K\alpha\beta=K\gamma_i\}|
\end{equation}
for each $1\le i\le 4$ (cf. \cite[p.52]{Shimura-book}). Let us compute $m(\gamma_i)$ for each $\gamma_i$. 

Let $J_I$ be the subset of $J$ consisting of the following elements
$$\alpha^s_I(x)=
\left(\begin{array}{ccc|ccc}
p\cdot 1_{s} &  & &&& \\
& p^2 & & &  &  \\
 & -p\cdot {}^t x & p\cdot 1_{n-1-s} &&& \\
 \hline 
&&& p\cdot 1_{s} &  & \\
&&&   & 1  & x\\
&&& &  &p\cdot 1_{n-1-s} 
\end{array}
\right),\ 0\le s\le n-2,\ x\in M_{1\times (n-1-s)}(\F_p)$$
and 
$\alpha^{n-1}_I=\diag(p^2,\overbrace{p,\ldots,p}^{n-1},1,\overbrace{p,\ldots,p}^{n-1})$. 

Similarly,  let $J_{II}$ be the subset of $J$ consisting of the following elements
$$\alpha^s_{II}(y,B_{21},B_{22},B_{33})=
\left(\begin{array}{ccc|ccc}
p\cdot 1_{s} &  & & 0_s& p\cdot {}^t B_{21} & 0_{s\times (n-1-s)} \\
-{}^t y & 1 &  & B_{21}& B_{22} &B_{23}  \\
 &  & p\cdot 1_{n-1-s} &  0_{(n-1-s)\times s} &p\cdot {}^tB_{23} & 0_{n-1-s} \\
 \hline 
&&& p\cdot 1_{s} &py  & \\
&&&   & p^2  &  \\
&&& &  &p\cdot 1_{n-1-s} 
\end{array}
\right)$$
where $1\le s\le n-1,\ y\in M_{s\times 1}(\F_p)$ and 
$B_{21}, B_{23}$, and $B_{22}$ run over 
$M_{1\times s}(\F_p),\ M_{1\times (n-1-s)}(\F_p)$, and $\Z/p^2\Z$ 
respectively.  
In addition, 
$$\alpha^0_{II}(C_{22},C_{23})=
\left(\begin{array}{cc|cc}
1 &  & C_{22} & C_{23} \\
& p\cdot 1_{n-1} &  p\cdot {}^tC_{23}  & 0_{n-1}  \\
 \hline 
& & p^2 &  \\
& &  & p\cdot 1_{n-1}   
\end{array}
\right),\ C_{22}\in \Z/p^2\Z,\ C_{23}\in M_{1\times (n-1)}(\F_p).$$
Finally, let $J_{III}$ be the subset of $J$ consisting of the following elements 
$$\alpha_{III}(B)=\left(\begin{array}{c|c}
p\cdot 1_n & B \\
\hline 
& p\cdot 1_n  
\end{array}
\right),\ B\in S_n(\F_p)\ \text{with $r_p(B)=1$}.
$$
\begin{lem}\label{KalphaK}For each $\alpha\in J$, 
$$K\alpha K=K\diag(1,\overbrace{p,\ldots,p}^{n-1},p^2,\overbrace{p,\ldots,p}^{n-1})K,
$$
and $${\rm vol}(K\diag(1,\overbrace{p,\ldots,p}^{n-1},p^2,\overbrace{p,\ldots,p}^{n-1})K)
=p\sum_{i=0}^{2n-1}p^i$$
where the measure is normalized as ${\rm vol}(K)=1$. 
\end{lem}
\begin{proof} 
Except for the case of type III, it follows from elementary divisor theory. 
For type III, it follows from \cite{Mac} that the action of $\GL_n(\F_p)$ on the set of 
all matrices of rank 1 in $S_n(\F_p)$ given by $B\mapsto {}^tXBX,\ X\in \GL_n(\F_p)$ and 
such a symmetric matrix $B$ has two orbits $O(\diag(1,\overbrace{0,\ldots,0}^{n-1}))$ and 
$O(\diag(g,\overbrace{0,\ldots,0}^{n-1}))$ where $g$ is a generator of $\F^\times_p$. 
The claim follow from this and elementary divisor theorem again.  

For the latter claim, it is nothing but $|J|$ and we may compute the number of each type. 
\end{proof}
\begin{remark}\label{WeylK}
Since $K=\Sp_{2n}(\Z_p)$ contains Weyl elements, 
\begin{eqnarray}
K\diag(1,\overbrace{p,\ldots,p}^{n-1},p^2,\overbrace{p,\ldots,p}^{n-1})K 
&=&K\diag(\overbrace{p,\ldots,p}^{i},1,\overbrace{p,\ldots,p}^{n-i-1},
\overbrace{p,\ldots,p}^{i},p^2,\overbrace{p,\ldots,p}^{n-i-1})K  \nonumber \\
&=&K\diag(\overbrace{p,\ldots,p}^{i},p^2,\overbrace{p,\ldots,p}^{n-i-1},
\overbrace{p,\ldots,p}^{i},1,\overbrace{p,\ldots,p}^{n-i-1})K \nonumber
\end{eqnarray}
for $0\le i\le n-1$.  
\end{remark}

Notice that $Kd_{n-1}(p)K=K(p^2\cdot d_{n-1}(p)^{-1})K$ 
where $d_{n-1}(p):=\diag(1,\overbrace{p,\ldots,p}^{n-1},p^2,\overbrace{p,\ldots,p}^{n-1})$. 
By definition and Lemma \ref{KalphaK} with Remark \ref{WeylK}, it is easy to see that 
\begin{eqnarray}
m(\gamma_i)&=&|\{\beta\in J\ |\ \gamma_i\beta^{-1}\in 
Kd_{n-1}(p)K\}| \nonumber \\
&=&|\{\beta\in J\ |\ \beta\cdot (p^2\cdot \gamma_i^{-1})\in 
Kd_{n-1}(p)K\}| \nonumber \\
&=&|\{\beta\in J\ |\  \text{$\beta\cdot (p^2\cdot \gamma_i^{-1})$ is $p$-integral and } r_p(\beta\cdot (p^2\cdot \gamma_i^{-1}))=1\}| \nonumber
\end{eqnarray}
(see \cite[p.52]{Shimura-book} for the first equality).

We are now ready to compute the coefficients. 
For $m(\gamma_1)$,  we observe the $p$-integrality. We see that only $\alpha^0_{II}(C_{22},C_{23})$ with $C_{22}=0$ and 
$C_{23}=0_{1\times(n-1)}$ can 
contribute there. Hence $m(\gamma_1)=1$. 

For $m(\gamma_2)$,  we observe the $p$-integrality and 
the rank condition. Then only $\alpha^0_{II}(0,0_{1\times (n-1)})$ and 
$\alpha^1_{II}(y,0,0,0_{1\times (n-2)}),\ y\in \F_p$  can do there. Hence $m(\gamma_2)=1+p$. 
For $m(\gamma_3)$, only $\alpha_{III}(B),\ B\in S_n(\F_p)$ with $r_p(B)=1$ contribute.
By Lemma \ref{DC}-(3), we have $m(\gamma_3)=p^n-1$. 

Finally, we compute $m(\gamma_4)$. 
Since $p^{-2}\gamma_4=I_4$, the condition is checked easily. 
All members of $J=J_I\cup J_{II}\cup J_{III}$ can contribute there. 
Therefore, we have only to count the number of each type. Hence, we have 
$$m(\gamma_4)=
\overbrace{1+p+\cdots+p^{n-1}}^{{\rm type}\ I}+
\overbrace{p^{n+1}+p^{n+2}+\cdots+p^{2n}}^{{\rm type}\ II}+
\overbrace{p^n-1}^{{\rm type}\ III}
=p\sum_{i=0}^{2n-1}p^i$$
as desired. Note that $m(\gamma_4)$ is nothing but the volume of 
$Kd_{n-1}(p)K$ (see Lemma \ref{KalphaK}). 

Recalling  $T_{p,n-1}:=pT(p,(0,\ldots,0,1))$, we have 
\begin{equation*}\label{final-form}
T(p,(0,\ldots,0,1))^2=\sum_{i=1}^4 m(\gamma_i) K(p^{-2}\gamma_i)K.
\end{equation*}
Note that
$$K(p^{-2}\gamma_1)K=T(p,(\overbrace{0,...,0}^{n-1},2)),\quad
K(p^{-2}\gamma_2)K= T(p,(\overbrace{0,...,0}^{n-2},1,1)),
$$
$$
K(p^{-2}\gamma_3)K= T(p,(\overbrace{0,...,0}^{n-1},1)),\quad 
K(p^{-2}\gamma_4)K= T(p,\overbrace{(0,...,0)}^{n})=K I_{2n}K.
$$ 
We can take $K$ back to $\G(N)$ without changing anything since $p\nmid N$. 
This proves Theorem \ref{main-appendix}.

\begin{remark} We would like to make corrections to \cite{KWY}. 
\begin{enumerate}
\item On page 356, line 1, $dxdy$ is missing in $\mu_\infty^{\rm ST}$. In [25, page 929, line 3], the same typo is repeated.
\item On page 362, line 12-13, $T_{2,p}^2$ should be a linear combination of 4 double cosets $KMK$, where $M$ runs over
$\diag(1,p^2,p^4,p^2), \diag(p,p,p^3,p^3), \diag(p,p^2,p^3,p^2), \diag(p^2,p^2,p^2,p^2)$. 
\item On page 362, the coefficient of $R_{p^2}$ should be $p^4+p^3+p^2+p=p\ds\sum_{i=0}^3p^i$ 
which is the volume of $\Sp(4,\Z_p)\diag(1,p^2,p^4,p^2)\Sp(4,\Z_p)$ explained in \cite[p.190, line -9 to -8]{RS}. 
\item On page 403, Lemma 8.1, the inequality $q(F)\geq N$ is not valid. Similarly, in page 405, Lemma 8.3, the inequality $q(F)\geq N$ is not valid. We need to consider newforms as in Section 5 of this paper. Then for a newform, we obtain the inequality $q(F)\geq N^{\frac 12}$ and $\log c_{\uk,N}\asymp \log N$ is valid as in Lemma \ref{logN} of this paper.
\item On page 409, line 10, we need to add $-2(G(\frac 32)+G(-\frac 12))$, in order to account for the poles of $\Lambda(s,\pi_F,{\rm Spin})$, and the contour integral is over $Re(s)=2$. So we need to add $O\left(\frac {|HE_\uk(N)^0|}{|HE_\uk(N)|}\right)$ in (9.3).
However, only CAP forms give rise to a pole, and the number of CAP forms in $HE_\uk(N)$ is $O(N^{8+\epsilon})$. So it is negligible. 

In the case of standard $L$-functions, the non-CAP and non-genuine forms which give rise to poles are: $1\boxplus \pi$, where $\pi$ is an orthogonal cuspidal representation of $GL(4)$ with trivial central character, or $1\boxplus\pi_1\boxplus \pi_2$, where $\pi_i$'s are 
dihedral cuspidal representations of $GL(2)$. In those cases, by Proposition \ref{trivial-case} and \cite[Theorem 2.9]{KWY1}, we can count such forms without extra conditions on $N$ in Proposition \ref{fixedV}. So our result is valid as it is written.
\end{enumerate}
\end{remark}

\end{document}